\documentclass[10pt,notitlepage]{article}
\usepackage{amsmath, amsthm, amssymb}
\usepackage{graphicx}
\usepackage[v2]{xy}
\xyoption{all}
\newtheorem{corollary}{Corollary}
\newtheorem{proposition}{Proposition}
\newtheorem{lemma}{Lemma}
\newtheorem{theorem}{Theorem}
\newtheorem{definition}{Definition}
\newtheorem{remark}{Remark}
\newtheorem{observation}{Observation}

\newcommand\wt{\widetilde}
\newcommand\ov{\overline}
\newcommand\inj{\rightarrowtail}
\newcommand\surj{\twoheadrightarrow}
\long\def\symbolfootnote[#1]#2{\begingroup%
\def\thefootnote{\fnsymbol{footnote}}\footnote[#1]{#2}\endgroup} 

\begin{document}
\title{{Polynomially bounded cohomology and the Novikov Conjecture}}
\date{}
\author{Crichton Ogle\\
Dept. of Mathematics\\The Ohio State University\\
\texttt{ogle@math.ohio-state.edu}}
\maketitle
\begin{abstract} Using techniques developed for studying polynomially bounded cohomology, we show that the assembly map for $K_*^t(\ell^1(G))$ is rationally injective for all finitely presented discrete groups $G$. This verifies the $\ell^1$-analogue of the Strong Novikov Conjecture for such groups. The same methods show that the Strong Novikov Conjecture for all finitely presented groups can be reduced to proving a certain (conjectural) rigidity of the topological cyclic chain complex $CC_*^t(H^{CM}_m(F))$ where $F$ is a finitely-generated free group and $H^{CM}_m(F)$ is the \lq\lq maximal\rq\rq Connes-Moscovici algebra associated to $F$.
\end{abstract}
\tableofcontents
\symbolfootnote[0]{2000 {\it Mathematics Subject Classification}. Primary 58B34; Secondary 18G10, 18G30, 18G35, 18G40, 18G60, 19K56, 46H25, 46L80, 46L87, 46M20, 55N35, 55T05, 58B34.}
\symbolfootnote[0]{{\it Key words and phrases}. Strong Novikov Conjecture, simplicial rapid decay algebras, polynomially bounded cohomology.}
\vskip.5in


\newpage
\section{Introduction}
\vskip.4in
The starting point for the work presented here is the following
\vskip.25in

{\bf\underbar{ Conjecture - Novikov, 1970}} (NC) {\it Let $M$ be a closed, compact, oriented $n$-dimensional manifold, ${\cal L}(M)$ its total Hirzebruch $L$-class, $[M]$ its fundamental homology class, and $\iota:M\to B\pi$. Then for every $[c]\in H^*(B\pi;\mathbb Q)$, the higher signatures
\[
Sign_c(M) := \left<{\cal L}(M)\iota^*(c),[M]\right>\in \mathbb Q
\]
are invariants of the oriented homotopy type of $M$}.
\vskip.25in

Let NC($\pi$) denote the NC is true for the discrete group $\pi$. Early work of Novikov and Lustig showed NC($\pi$) was true for $\pi = \mathbb Z^n$. However, further progress required a reformulation. Beginning with
$\mathbb{L}_{\bullet}(R) =$ the (symmetric/quadratic) $L$-theory spectrum of a ring with involution $R$, one has the Mishchenko-Quinn-Ranicki-Wall Assembly map
\[
{\cal A}^{\mathbb L}_*(\pi): H_*(B\pi;\mathbb{L}_{\bullet}(\mathbb Z))\to L_*(\mathbb Z[\pi]) = \pi_*(\mathbb{L}_{\bullet}(\mathbb Z[\pi]))
\]
\vskip.25in

{\bf\underbar{The Mishchenko-Ranicki-Wall reformulation}} (MRW-NC) {\it The assembly map ${\cal A}^{\mathbb L}_*(\pi)$ is rationally injective}.
\vskip.25in

Let MRW-NC($\pi$) denote this conjecture for the particular group $\pi$. Wall [W1] showed that for any discrete group $\pi$, MRW-NC($\pi$) implies NC($\pi$). Subsequently, Mishchenko and Ranicki showed the converse - that NC($\pi$) implies MRW-NC($\pi$) - when $\pi$ is finitely presented\footnote{If NC($\pi$) is true for all finitely presented groups, it is true for all groups, and S. Weinberger has pointed out that validity of MRW-NC($\pi$) for all finitely presented groups $\pi$ implies MRW-NC for all discrete groups, via a suitable direct limit argument. It is not clear if such a statement holds for the Banach algebra versions of the conjecture stated below.}.
\vskip.25in

For a given finitely presented group $\pi$, NC($\pi$) will refer either to the original conjecture posed by Novikov, or the equivalent conjecture MRW-NC($\pi$). The first major breakthrough
using this formulation of NC was achieved in the mid 1970's by Mishchenko, who showed NC($\pi$) is true whenever $\pi$ is the fundamental group of a closed, compact (oriented) manifold of negative curvature. A number of years later, another major step was taken by G. Kasparov. He formulated the so-called Strong Novikov Conjecture, which for a given discrete group $\pi$ states
\vskip.25in

{\bf\underbar{SNC($\pi$) - Kasparov}} {\it Let $C^*(\pi)$ be a faithful $C^*$-completion of $\mathbb C[\pi]$. Then the assembly map in topological $K$-theory
\[
{\cal A}_*(\pi): H_*(B\pi;\mathbb{K}^t_{\bullet}(\mathbb C))\to K^t_*(C^*(\pi)) 
\]
is rationally injective.}
\vskip.25in
It is easily seen that SNC($\pi$) implies NC($\pi$)  using the fact that the topological Witt groups and topological $K$-groups of a $C^*$-algebra are naturally isomorphic, but this stronger version has other consequences as well, hence its name. Kasparov then showed in [GK] that SNC($\pi$) is true (for the reduced $C^*$-algebra $C^*_r(\pi)$) whenever $\pi$ is a cocompact discrete subgroup of a virtually connected Lie Group. This stronger version of the conjecture has been verified for a number of other classes of groups; in particular, for word-hyperbolic groups [CM], and amenable groups [HR].
\vskip.25in

Inside of $C^*(\pi)$ lies the smaller convolution algebra $\ell^1(\pi)$ of $\ell^1$-functions on $\pi$, and one can formulate an $\ell^1$-analogue of Kasparov's conjecture:
\vskip.25in

{\bf\underbar{$\ell^1$-NC($\pi$)}} {\it The assembly map in topological $K$-theory
\[
{\cal A}_*(\pi): H_*(B\pi;\mathbb{K}^t_{\bullet}(\mathbb C))\to K^t_*(\ell^1(\pi)) 
\]
is rationally injective.}
\vskip.25in

There is a factorization
\[
{\cal A}_*(\pi): H_*(B\pi;\mathbb{K}^t_{\bullet}(\mathbb C))\to K^t_*(\ell^1(\pi))\to K^t_*(C^*(\pi)) 
\]
by which one sees that SNC($\pi$) implies $\ell^1$-NC($\pi$); or equivalently, that any counterexample to $\ell^1$-NC($\pi$) would provide a counterexample to SNC($\pi$). 
\vskip.25in

Until recently the class of groups for which these last two conjectures were known was essentially the same. In [O1], and independently in [M1], it was shown that groups which are synchronously combable in polynomial time are {\it $\cal P$-isocohomological}, implying every (complex) group cohomology class is represented by a cocycle of polynomial growth (this is the (PC) condition of [CM]). By [PJ1] and [PJ2], any group satisfying this condition also satisfies $\ell^1$-NC($\pi$). In [JOR2] the isocohomological results of [O1] were extended, yielding many more groups which are $\cal P$-isocohomological - and which therefore also satisfy $\ell^1$-NC($\pi$) - for which the original Novikov Conjecture is not currently known. However, as is also shown in [JOR2], there are many groups (including even solvable groups) which do not satisfy condition (PC). In fact, for finitely presented groups which are not of type $FP^{\infty}$, it is easily seen that there exist cohomology classes which are completely unbounded with respect to a fixed word-length function on $G$. So the $C$-extendability methods of [CM], or their $\ell^1$-analogue, clearly do not apply to such cohomology classes. And so for groups with rational homology detectable only by such classes, verifying even $\ell^1$-NC($\pi$) via classical cocycle extension techniques and cyclic theory is problematic. Our main result is that these obstructions can be bypassed, at least for the $\ell^1$-algebra. Precisely, we show
\vskip.25in

{\bf\underbar{Theorem A}} The conjecture $\ell^1$-NC($\pi$) is true for all finitely presented groups $\pi$.
\vskip.25in

The proof of Theorem A follows by a detailed analysis of simplicial rapid decay algebras; simplicial analogues of the rapid decay algebras introduced by Jolissaint in [PJ1], [PJ2]. We begin in section 2 with some preliminaries concerning Hochschild and cyclic homology. We also explain (using techniques from [O2]) how Theorem A follows from verifying the injectivity of the {\it restricted} assembly map
\[
H_*(B\pi;\mathbb Q)\to K_*^t(\ell^1(\pi))\otimes\mathbb Q
\]
 Although not essential, this reduction simplifies various considerations later on.
\vskip.25in

In section 3.1, we construct the $\ell^1$-rapid decay algebra associated to a p-bounded simplicial group (as defined in [O1]). In sections 3.2 and 3.3, we give the analogous extension to p-bounded simplicial groups of the Connes-Moscovici algebra, and relate it to the simplicial $\ell^2$-rapid decay (aka Jolissaint) algebra when the group is degreewise rapid decay (e.g., free). We also define a \lq\lq maximal\rq\rq Connes-Moscovici algebra, which (like its $\ell^1$-analogue) has the advantage of producing a functor $(f.p.groups)\mapsto (Fr\acute{e}chet\ algebras)$ from the category of finitely presented groups to the category of Fr\'echet algebras, which can therefore be naturally extended to a simplicial framework.
\vskip.25in

In section 4, we review the spectral sequences arising in Hochschild and cyclic (co-)homology for the simplicial algebras considered in the previous section, beginning with the simplicial group algebra. The prototypical spectral sequences are those for the simplicial group algebra $\mathbb C[\Gamma\hskip-.03in .]$ when $\Gamma\hskip-.03in .$ is a free simplicial resolution of $\pi$. In section 4.2 we introduce the key ingredient in the proof of the Theorem, namely the construction of a local Chern character associated to an arbitrary integral homology class $x\in H_n(B\pi;\mathbb Q)$. This Chern character, defined for both $K_*^t(\ell^1(\pi))$ and $K_*^t(C^*(\pi))$, is taylor-made to detect the image of $x$ under the rationalized assembly map, and results from a combination of Baum's retopologization Theorem, the Chern character for fine topological algebras constructed by Tillmann in [T1] (results discussed in the first appendix), and the \lq\lq rigidity\rq\rq\  results of Goodwillie [G1]. In section 4.3, we complete the proof of Theorem A by a careful analysis of the spectral sequence for the topological cyclic homology of the $\ell^1$-rapid decay algebra constructed in the previous section. Finally, in section 4.4, we prove a weaker injectivity result for the maximal Connes-Moscovici algebra $H^{CM}_m(G)$. The main result of this section, similar in spirit to the main theorem of [O2], is that an injectivity result holds for the topological \underbar{Hochschild} homology groups of a certrain simplicial Fr\'echet algebra, and that if the same result holds upon passage (via the $I$ map) from the topological Hochschild to the topological cyclic homology groups of this algeba, this would imply the Strong Novikov Conjecture for all finitely presented groups.
\vskip.25in

This paper incorporates a number of ideas from the multiply-revised preprint [O0], which was also a source for [O1] and [O2].  In some ways this paper represents the final iteration of what was begun in [O0], at least as far as the $\ell^1$-group algebra is concerned.
\vskip.25in

 I would like to thank Dan Burghelea for his support and encouragement during the evolution of this work, as well as the various (anonymous) referees whose criticisms over the years, applied to various stages of [O0] (including the current version), have proved to be quite valuable.
\newpage


\section{Preliminaries}
\vskip.4in

In this paper all vector spaces (both algebraic and topological), algebras, tensor products, Hom groups, are over $\mathbb C$. Moreover, homology and cohomology will always be with coefficients in $\mathbb C$, unless indicated otherwise by notation. 
\vskip.2in

For an algebra $A$, we will write $CH_*^a(A)$ resp.\ $CC_*^a(A)$ for the algebraic Hochschild resp.\ cyclic complex of $A$, with $HH_*^a(A)$ resp.\ $HC_*^a(A)$ denoting their corresponding homology groups. In a similar vein, we set $CH^*_a(A) := Hom(CH_*^a(A),\mathbb C)$,  $CC^*_a(A) := Hom(CC_*^a(A),\mathbb C)$ with associated cohomology groups denoted by $HH^*_a(A)$ resp.\ $HC^*_a(A)$. More generally, for a vector space $V$ we will denote $Hom(CH_*^a(A),V)$ resp.\ $Hom(CC_*^a(A),V)$ by $CH^*_a(A;V)$ resp. $CC^*_a(A;V)$, and the corresponding cohomology groups by $HH^*_a(A;V)$ resp.\ $HC^*_a(A;V)$. In the case of Hochschild homology and cohomology, this should not be confused with the notion of the Hochschild (co-)homology of $A$ with coefficients in an $A$-bimodule $M$. We assume familiarity with the traditional constructions and properties of these theories; an appropriate reference is [L]. In the event the subscript (resp.\ superscript) \lq\lq a\rq\rq is omitted, it will be understood we are referring to algebraic Hochschild or cyclic (co)-homology.
\vskip.2in

If $A$ is a locally convex topological algebra, then the Hochschild and cyclic complexes for $A$ may be formed using the projectively complete tensor product, resulting in topological chain complexes $CH_*^t(A)$ resp. $CC_*^t(A)$. If $W$ is a topological vector space, one also has $CH^*_t(A;W) := Hom^{cont}(CH_*^t(A),W)$, $CC^*_t(A;W) := Hom^{cont}(CC_*^t(A),W)$. If $D_* = (D_*,d_*)$ is a topological chain complex, the \underbar{unreduced} resp.\ \underbar{reduced} homology of $D_*$ is given by
\begin{gather*}
H_n^t(D_*) := ker(d_n)/im(d_{n+1})\\
\overline{H}_n^t(D_*) := ker(d_n)/\overline{im(d_{n+1})}
\end{gather*}
One has similarly defined groups in cohomology:
\begin{gather*}
H_t^n(Hom^{cont}(D_*,W)) := ker(\delta^n)/im(\delta^{n-1})\\
\overline{H}_t^n(Hom^{cont}(D_*,W)) := ker(\delta^n)/\overline{im(\delta^{n-1})}
\end{gather*}
Although there is no Universal Coefficient Theorem in the topological setting, there is still a natural homomorphism
\begin{gather*}
H_t^n(Hom^{cont}(D_*,W))\to Hom^{cont}(\overline{H}^t_*(D_*),W),\,\,\text{indicated by}\\
H_t^n(Hom^{cont}(D_*,W))\ni [f]\mapsto [f]_*: \overline{H}^t_*(D_*)\to W
\end{gather*}
\vskip.2in

When $A$ is equipped with a topology,  we will write $F^a_*(A)$ resp. $F^*_a(A)$ for $F^t_*(A^{\delta})$ resp. $F^*_t(A^{\delta})$, where $A^{\delta}$ denotes $A$ with the discrete topology and $F$ represents one of the (co-)chain complex functors or associated (co-)homology functors discussed above. For any choice of $F$, there are natural transformations $F^a_*(_-)\to F^t_*(_-)$, $F^*_t(_-)\to F^*_a(_-)$. Topologies on $A$ are assumed to be continuous over $\mathbb C$. For a given $A$, the \underbar{fine topology} on $A$ will be denoted by $A^f$. For any other continuous topology $A^T$ on $A$, the identity map on elements determines a continuous map $A^f\to A^T$. Because the algebraic tensor product is complete in the fine topology, there are isomorphisms of graded vector spaces
\begin{gather*}
F^a_*(A)\cong F^t_*(A^f)\cong \ov{F}^t_*(A^f)\\
F_a^*(A)\cong F_t^*(A^f)\cong \ov{F}_t^*(A^f)
\end{gather*}
for $F(_-) = HH(_-), HC(_-)$.
\vskip.2in

For the complex group algebra $\mathbb C[\pi]$, there are well-known decompositions of $CH_*(\mathbb C[\pi])$ and $CC_*(\mathbb C[\pi])$ as direct sums of subcomplexes, indexed on $<\pi> =$ the set of conjugacy classes of $\pi$, which induce corresponding decompositions in homology:
\begin{gather*}
CH_*(\mathbb C[\pi])\cong \bigoplus_{<x>\in <\pi>} CH_*(\mathbb C[\pi])_{<x>}\\
CC_*(\mathbb C[\pi])\cong \bigoplus_{<x>\in <\pi>} CC_*(\mathbb C[\pi])_{<x>}\\
HH_*(\mathbb C[\pi])\cong \bigoplus_{<x>\in <\pi>} HH_*(\mathbb C[\pi])_{<x>}\\
HC_*(\mathbb C[\pi])\cong \bigoplus_{<x>\in <\pi>} HC_*(\mathbb C[\pi])_{<x>}
\end{gather*}
In cohomology, one has similar isomorphisms, with a direct sum of complexes replaced by a direct product of cocomplexes. The projection onto the summand indexed by $<x>$, on both the (co-)chain or (co-)homology level, will be denoted by $p_{<x>}$. In this paper, we will be primarily concerned with the projection $p_{<1>}$. For the $\ell^1$-rapid decay algebra $H^{1,\infty}_L(\pi)$ associated to a discrete group with word-length $(\pi,L)$, viewed as a Frech\'et algebra, one has a similar decomposition on the chain level (cf. [RJ], [JOR1])
\begin{gather*}
CH^t_*(H^{1,\infty}(\pi))\cong \widehat{\bigoplus}_{<x>\in <\pi>} CH^t_*(H^{1,\infty}(\pi))_{<x>}\\
CC^t_*(H^{1,\infty}(\pi))\cong \widehat{\bigoplus}_{<x>\in <\pi>} CC^t_*(H^{1,\infty}(\pi))_{<x>}
\end{gather*}
This produces projection maps, also denoted $p_{<x>}$, for both topological (co-)chains, as well as topological (co-)homology, both reduced and unreduced. On the (co-)chain level and for the reduced theories, these projection maps are continuous in the induced topology.
\vskip.2in

In this paper, the category of finitely presented groups will play a central role. In order to maintain control on the size of this category, we fix a countable set ${\cal R}_U := \{x_{\alpha}\}_{\alpha\in {\cal I}}$, where $\cal I$ is a countable indexing set. The objects of the category $(f.p.groups)$ are finitely presented groups $G = <R,W>$ on a generating set $R$ which is required to be a finite subset of ${\cal R}_U$, where the presentation is included as part of the data associated to the object. The morphisms of $(f.g.groups)$ are group homomorphisms. Each object in $(f.g.groups)$ is equipped with a standard word-length function $L_{st}$. As the groups are finitely presented, restriction of generating sets to finite subsets of the countable universal collection of generators ${\cal R}_U$ results in $Obj(f.p.groups)$ being a countable set. Additionally, the fixed word-length function shows that each $Hom$ set $Hom(G,G')$ is countable as well (these facts are used below in the construction of the \lq\lq maximal\rq\rq analogue of the Connes-Moscovici algebra).
\vskip.2in

For a Banach algebra $A$, let $\underline{\underline{K}}^t(A)$ denote the topological $K$-theory spectrum for $A$. Let $\underline{\underline{S}}$ denote the sphere spectrum (i.e., the reduced suspension spectrum of $S^0$). Then $1\in K_0^t(\mathbb C)$ correpsonds to a map of spectra
$\underline{\underline{S}}\to \underline{\underline{K}}^t(\mathbb C)$; precomposition of this \lq\lq unit\rq\rq with the full assembly map yields the restricted assembly map on the level of spectra:
\begin{equation}\label{eqn:restricted}
B\pi_+\wedge \underline{\underline{S}}\to B\pi_+\wedge \underline{\underline{K}}^t(\mathbb C)\to \underline{\underline{K}}^t(\ell^1(\pi))
\end{equation}

\begin{lemma} The full asssembly map $B\pi_+\wedge \underline{\underline{K}}^t(\mathbb C)\to \underline{\underline{K}}^t(\ell^1(\pi))$ is rationally injective on homotopy groups iff the restricted assembly map is so.
\end{lemma}

\begin{proof} In [O2, Cor. 4.5] we proved this result with $C^*(\pi)$ in place of $\ell^1(\pi)$. The proof for the $\ell^1$-group algebra follows by exactly the same argument leading up to [O2, Thm. 4.4] once one makes the following observations:
\begin{itemize}
\item If $F$ is a free group, the assembly map $BF_+\wedge \underline{\underline{K}}^t(\mathbb C)\to \underline{\underline{K}}^t(\ell^1(F))$ is a rational homotopy equivalence (in fact, it is an integral equivalence).
\item Given a Banach algebra $A$ and a closed ideal $I\subseteq A$, the short-exact sequence $I\inj A\surj A/I$ induces a homotopy fibration sequence of spectra: 
\[
\underline{\underline{K}}^t(I)\to \underline{\underline{K}}^t(A)\to \underline{\underline{K}}^t(A/I)
\]
\item An augmented free simplicial resolution $\Gamma\hskip-.03in .^+$ of $\Gamma_{-1} = \pi$ induces a filtration on the homotopy groups of both $B\pi_+\wedge \underline{\underline{K}}^t(\mathbb C)$ and $\underline{\underline{K}}^t(\ell^1(\pi))$, with respect to which the assembly map is filtration-preserving. The filtration of $KU_*(B\pi)$ rationally corresponds to the skeletal filtration, while the filtration of $K_*^t(\ell^1(\pi))$ is as defined in [O2, (1.5)].
\end{itemize}
\end{proof}
\newpage


\section{Simplicial rapid decay algebras}
\vskip.4in

\subsection{$\ell^1$-rapid decay algebras}
\vskip.3in

 Let $\ell^1(\pi)$ be the convolution algebra of $\ell^1$-functions on $\pi$, with standard $\ell^1$-norm. 
For a real number $s> 0$, $H^{1,s}_L(\pi)$ is the Banach algebra of functions
$f: \pi \to \mathbb C $ satisfying the inequality
$\nu_{1,s}(f) < \infty$, where

\begin{equation}\label{eqn:nus1}
\nu_{1,s}(f) = \sum\limits_{h\in\pi}\, |f(h)|(1+L(h))^{s}
\end{equation}

For each $s$ and $L$, $ H^{1,s}_L(\pi)$ is the completion of the group algebra
$\mathbb C[\pi]$ with respect to the semi-norm $\nu_{1,s}$. Let

\begin{equation}\label{eqn:l1inf}
H^{1,\infty}_L(\pi) := \underset s{\bigcap}\, H^{1,s}_L(\pi)
\end{equation}

The collection of semi-norms $\{\nu_{1,n}\}_{n\in\mathbb N}$ give
$H^{1,\infty}_L(\pi)$ the structure of a Fr\'echet space. Unlike the $\ell^2$-case discussed below,
the \underbar{$\ell^1$-rapid decay algebra} $H^{1,\infty}_L(\pi)$ is always a subalgebra of $\ell^1(\pi)$  which is closed in $\ell^1(\pi)$ under holomorphic functional calculus [PJ1]. Moreover, it is functorial with respect to group homomorphisms $\phi:(\pi,L)\to (\pi',L')$ which are polynomially bounded with respect to word-length (i.e., there is a polynomial $p$ for which $L'(\phi(g))\le p(L(g))$ for all $g\in  \pi$). 
\vskip.2in

Let $(f.p.groups)$ be the category of finitely-presented groups defined above. If $G$ is finitely presented, any two word-length functions on $G$ are linearly equivalent, hence linearly equivalent to the standard word-length function $L^G_{st}$ on $G$. For this reason, we will assume for the remainder of the paper that the word-length function on a finitely presented group is always $L_{st}$. Observing that any homomorphism from a finitely-generated group $G$ is polynomially (in fact, linearly) bounded with respect to $L^G_{st}$, we conclude

\begin{proposition}\label{prop:l1-functor} The association $G\mapsto H^{1,\infty}_L(G)$ defines a functor
\[
H^{1,\infty}_L(_-): (f.p.groups)\to (Fr\acute{e}chet\ algebras)
\]
where $(Fr\acute{e}chet\ algebras)$ denotes the category of Fr\'echet algebras and continuous Fr\'echet algebras homomorphisms.
\end{proposition}
\vskip.5in


\subsection{The Connes-Moscovici algebra of a discrete group}
\vskip.3in

We recall the construction of [CM]. As before, $\pi$ will
denote a countable discrete group equipped with a word-length function $L$. 
By convention, $\ell^2(\pi)$ is the Hilbert space of complex-valued
$\ell^2$-functions on $\pi$, with ${\cal L}(\ell^2(\pi))$ the space of
bounded operators on $\ell^2(\pi)$. On $\ell^2(\pi)$ one has
$D_L : \ell^2(\pi)\to \ell^2(\pi)$ given on basis elements by
$D_L(\delta_g) = L(g)\delta_g$. This operator in turn defines an unbounded operator
$\partial_L : {\cal L}(\ell^2(\pi))\to {\cal L}(\ell^2(\pi))$ given by
\[
\partial_L(M) = ad(D_L)(M) = D_L\circ M - M\circ D_L
\]

Following [CM], we define the Connes-Moscovici algebra associated with the
pair $(\pi,L)$ as

\begin{definition} $H_L^{CM}(\pi) = \{a\in C^*_r(\pi)\,|\, \partial_L^k(a)\in {\cal L}(\ell^2(\pi))\,\text{ for all } k\}$.
\end{definition}

Note that $H_L^{CM}(\pi) = \displaystyle\left(\bigcap_{i=1}^\infty Domain(\partial_L^k)\right)\cap C^*_r(\pi)$. Since
$$
\partial_L^k(ab) = \sum_{i=1}^k \binom ki \partial_L^i(a)\partial_L^{k-i}(b)
$$
it follows $H_L^{CM}(\pi)$ is an algebra. It is also easy to see that
$H_L^{CM}(\pi)$ contains the group algebra.

\begin{proposition}\label{prop:cm} For all word-length functions
$L$ and groups $\pi$, $H_L^{CM}(\pi)$ is a dense subalgebra of
$C^*_r(\pi)$ closed under holomorphic functional calculus.
\end{proposition}

(A detailed proof of this Proposition is given in [RJ]).
\vskip.2in

Define semi-norms on $\mathbb C[\pi]$ by

\begin{equation}\label{eqn:etam}
\eta_m(a) = \sum_{i=0}^m \binom mi \|\partial_L^i(a)\|
\end{equation}

where $\|_-\|$ is the norm on $C^*_r(\pi)$ and $\partial_L^0 = Id$.
Using standard properties of $\|_-\|$ we have

\begin{gather}\label{eqn:etaprod}
\eta_m(ab)\le \eta_m(a)\eta_m(b)\\
\eta_m(a)\le \eta_n(a)\quad\text{when}\, m\le n
\end{gather}

It is observed in [RJ] that $H_L^{CM}(\pi)$ is complete in the
semi-norms $\{\eta_m\}_{m\ge 0}$. As $\mathbb C[\pi]$ is dense in
$H_L^{CM}(\pi)$ in the topology induced by these semi-norms, we may
alternatively describe $H_L^{CM}(\pi)$ as the Fr\'echet algebra
completion of $\mathbb C[\pi]$ with respect to $\{\eta_m\}_{m\ge 0}$.
\vskip.2in

We consider a related completion of the group algebra $\mathbb C[\pi]$. For
each real number $s>0$, $H^{2,s}_L(\pi)$ is the Hilbert space of functions
$f: \pi \to \mathbb C $ satisfying the inequality
$\nu_{2,s}(f) < \infty$, where

\begin{equation}\label{eqn:nus}
\nu_{2,s}(f) = (<f,f>_{2,s,L})^{1/2}\,,\qquad
<f,g>_{2,s,L}
= \sum\limits_{h\in\pi}\, f(h)\overline{g(h)}(1+L(h))^{2s}
\end{equation}

For each $s$ and $L$ , $ H^{2,s}_L(\pi)$ is the completion of the group algebra
$\mathbb C[\pi]$ with respect to the semi-norm $\nu_{2,s}$. Let

\begin{equation}\label{eqn:l2inf}
H^{2,\infty}_L(\pi) := \underset s{\bigcap}\, H^{2,s}_L(\pi)
\end{equation}

The collection of semi-norms $\{\nu_{2,n}\}_{n\in\mathbb N}$ give
$H^{2,\infty}_L(\pi)$ the structure of a Fr\'echet space, as $H^{2,t}_L(\pi)\subset
H^{2,s}_L(\pi)$ for $0< s< t$. If there exist numbers $C,s$ such that

\begin{equation}\label{eqn:rd}
\|f\| \le C \nu_s(f)\qquad\text{for all }f\in \mathbb C[\pi]
\end{equation}

where $\|f\|$ denotes the reduced $C^*$-norm of $f$, then $H^{2,t}_L(\pi)$ is
a subspace of $C^*_r(\pi)$ for $t > s$ , and $H^{2,\infty}_L(\pi)$ is a Fr\'echet
subalgebra of $C^*_r( \pi)$ . This property was first shown to hold for
finitely-generated free groups by Haagarup [H], and later for
finitely-generated hyperbolic groups by Jolissaint [PJ1], [PJ2] and P. de la
Harpe [dH]. A group with word-length function $(\pi,L)$ satisfying
condition (\ref{eqn:rd}) is referred to as {\it{Rapid Decay}}, or RD. In general, $H^{2,\infty}_L(\pi)$
contains $H^{CM}_L(\pi)$ whether or not $\pi$ is RD [CM], and when
it is, the two are equal [RJ]. We will need a refined version of
this last result. Call a word-length function $L$ {\it{nice}} if
$L$ does not take any values in the open interval $(0,1)$.

\begin{lemma}\label{lemma:bound} Let $\pi$ be a discrete group with
nice word-length function $L$. Then for each integer $k\ge 0$ and
$\varphi\in\mathbb C[\pi]$
$$
\nu_{2,k}(\varphi)\le 2k\eta_{2k}(\varphi)
$$
If in addition there exists $s, C > 0$ for which (\ref{eqn:rd}) holds above
($\pi$ is RD), then
$$
\eta_k(\varphi)\le 2^kC\nu_{2,k+s}(\varphi)
$$
\end{lemma}

\begin{proof} The first inequality follows from

\begin{gather*}
\nu_{2,k}(\varphi)^2
= \underset {g\in\pi}{\sum} |\varphi(g)|^2(1 + L(g))^{2k}\\
= \underset {g\in\pi}{\sum}
|\varphi(g)|^2\left(\sum_{i=0}^{2k}\binom{2k}{i}L(g)^i\right)\\
= \sum_{i=0}^{2k}\binom{2k}{i}\left(\underset {g\in\pi}{\sum} |\varphi(g)|^2
L(g)^i\right)\\
\le \sum_{i=0}^{2k}\binom{2k}{i}\left(\underset {g\in\pi}{\sum} |\varphi(g)|^2
L(g)^{2i}\right)\\
\le \sum_{i=0}^{2k}\binom{2k}{i} \|\partial^i_L(\varphi)\|^2_2\\
\le \sum_{i=0}^{2k}\eta_{2k}(\varphi)^2\\
= 2k\eta_{2k}(\varphi)^2
\end{gather*}

Now suppose $\pi$ is RD and that $s, C > 0$ have been chosen for which
(\ref{eqn:rd}) holds. Then as in the proof of [RJ, Th. 1.3] one has
\[
\|\partial^m_L(\varphi)\|\le C\nu_{2,m+s}(\varphi)
\]
and the second inequality above follows from

\begin{gather*}
\eta_k(\varphi)
= \sum_{i=0}^k \binom{k}{i} \|\partial_L^i(\varphi)\|\\
\le C\sum_{i=0}^k \binom{k}{i} \nu_{2,i+s}(\varphi)\\
\le C\sum_{i=0}^k \binom{k}{i} \nu_{2,k+s}(\varphi)\\
= 2^k C\nu_{2,k}(\varphi)
\end{gather*}

These explicit inequalities also prove the two semi-norm
topologies are the same when $\pi$ is RD.
\end{proof}
\vskip.2in

 Observe that for all $(\pi,L)$, there are inclusions
\[
\mathbb C[\pi]\hookrightarrow H^{1,\infty}_L(\pi) \hookrightarrow H^{CM}_L(\pi)
\hookrightarrow H^{2,\infty}_L(\pi)
\]
with the first two inclusions being algebra homomorphisms, and the last two being continuous Frech\'et space maps.
\vskip.5in


\subsection{Simplicializing the Connes-Moscovici construction}
\vskip.3in

The construction of both $H^{CM}_L(\pi)$ and $H^{2,\infty}_L(\pi)$ lacks
functoriality with respect to group surjections, even those that are
p-bounded. The following proposition provides the method for extending the
Connes-Moscovici construction to simplicial or augmented simplicial
groups. To properly state it, we note that if $\{\eta'_i\}$ is a countable collection
of semi-norms on $\mathbb C[\pi]$, then the completion of $\mathbb C[\pi]$ in
these semi-norms produces a Fr\'echet space. If each semi-norm is \underbar {submultiplicative}, that
is,
\[
\eta'_i(ab) \le \eta'_i(a)\eta'_i(b)\quad \forall a,b\in \mathbb C[\pi]
\]
then the completion is a Fr\'echet algebra. In particular, the semi-norms
$\eta_m$ satisfy this property by (\ref{eqn:etaprod}).

\begin{proposition}\label{prop:map} Suppose $f : \pi_1\to \pi_2$ is
a group homomorphism. Let $S_i$ be a countable collection of
submultiplicative semi-norms on $\mathbb C[\pi_i]$, $i = 1,2$. Suppose also
that for each $\eta\in S_2$, $f^*\eta := \eta\circ f\in S_1$. Then the group algebra
homomorphism induced by $f$ extends to a continuous homomorphism of
Fr\'echet algebras
\[
\hat f : {\cal H}_{S_1}(\pi_1)\to {\cal H}_{S_2}(\pi_2)
\]
where ${\cal H}_{S_i}(\pi_i)$ denotes the completion of $\mathbb C[\pi_i]$ with
respect to the collection of semi-norms $S_i$.
\end{proposition}

\begin{proof} As noted, the completions are Fr\'echet algebras. So it will
suffice to prove that $f$ extends over the completions. Let $x\in {\cal
H}_{S_1}(\pi_1)$ be represented by the Cauchy sequence $\{x_i\}$ and
$\eta_2$ be any semi-norm in $S_2$. Then $\{x_i\}$ converges in the
semi-norm $\eta_2\circ f$, which is exactly saying that $\{f(x_i)\}$
converges in the semi-norm $\eta_2$. If $x$ is represented by another
sequence $\{x'_i\}$, then $\{x_i - x'_i\}$ converges to zero in the
semi-norm $\eta_2\circ f$, whence $\{f(x_i - x'_i) = f(x_i) - f(x'_i)\}$
converges to zero in the semi-norm $\eta_2$.
\end{proof}

In particular, if we start with a collection of semi-norms $T_1$ on $\mathbb
C[\pi_1]$, we can enlarge $T_1$ to a set $S_1$ which contains all
semi-norms of the form $f^*(\eta) = \eta\circ f$, $\eta\in S_2$. Then the
above proposition applies, and $f$ extends.
\vskip.2in

An \underbar{augmented simplicial group with word-length}
$(\Gamma\hskip-.03in .^+,L\hskip-.015in .^+)$ consists of an augmented simplicial group $\Gamma\hskip-.03in .^+$, with $L_n$ a word-length function on $\Gamma_n$ for all $n\ge -1$. The augmented simplicial group is \underbar{p-bounded} if all face and degeneracy maps, including the augmentation map $\varepsilon:\Gamma_0\twoheadrightarrow \Gamma_{-1}$, are polynomially bounded with respect to the word-length functions $\{L_n\}_{n\ge -1}$. Let $\Gamma(\varepsilon)_0 := \ker(\varepsilon)$, $\Gamma(\varepsilon)_n := \ker(\varepsilon_n = \varepsilon\circ\partial_0^{(n)}:\Gamma_n\to \Gamma_{-1})$. Then $\Gamma(\varepsilon)\hskip-.015in .$ is a simplicial subgroup of simplicial group $\Gamma\hskip-.03in . = \{\Gamma_n\}_{n\ge 0}$, and $\Gamma\hskip-.03in .^+$ is a resolution if $\pi_*(\Gamma(\varepsilon)\hskip-.015in .) = 0,*\ge 0$. Finally, $(\Gamma\hskip-.03in .^+,L\hskip-.015in .^+)$ is a \underbar{type $P$ resolution} if i) $(\Gamma_n,L_n)$ is a countably generated free group with $\mathbb N$-valued word-length metric $L_n$ induced by a proper function on the set of generators for $\Gamma_n$ for all $n\ge 0$, and ii) $\Gamma(\varepsilon)\hskip-.015in .$, viewed as a simplicial set, admits a simplicial contraction $s\hskip-.015in .' = \{s'_{n+1}:\Gamma(\varepsilon)_n\to \Gamma(\varepsilon)_{n+1}\}_{n\ge 0}$ which is polynomially bounded in each degree. Every countable discrete group $\pi$ admits a type $P$ resolution; moreover any p-bounded simplicial group $(\Gamma\hskip-.03in .,L\hskip-.015in .)$ includes as a simplicial subgroup of a (larger) type $P$ resolution $(\widetilde{\Gamma\hskip-.03in .}^+,\widetilde{L\hskip-.015in .}^+)$ where $\widetilde{\Gamma}_{-1} = \pi_0(\Gamma\hskip-.03in .)$ [Appendix, O1]. In fact, starting with $\pi$, the resolution can always be constructed so that the face and degeneracy maps, the augmentation $\varepsilon$ and the simplicial contraction $s\hskip-.015in .'$ are all linearly bounded.
\vskip.2in

Any augmented simplicial group $\Gamma\hskip-.03in .^+$ may be viewed as
a covariant functor $\Gamma\hskip-.03in .^+ : (\Delta^+)^{op}\to (gps)$, where $\Delta^+$
denotes the augmented simplicial category. For each $m,n\ge -1$, let
$S_{m,n}$ denote the image under $\Gamma\hskip-.03in .^+$ of the morphism set
$Hom_{(\Delta^+)^{op}}([m],[n])$; this is the set of all
homomorphisms from $\Gamma_m$ to $\Gamma_n$ which occur as an iterated
composition of face and degeneracy maps. For each $n\ge -1$, let
$\{\eta_{n,j}\}_{j\ge 0}$ resp. $\{\nu_{2,n,j}\}_{j\ge 0}$ denote the
collection of semi-norms associated with the group with word-length
$(\Gamma_n,L_n)$ as defined in (\ref{eqn:etam}) resp. (\ref{eqn:nus}). Finally we set

\begin{equation}\label{eqn:set}
S_m = \coprod_{n\ge -1}\coprod_{\lambda\in
S_{m,n}}\{\lambda^*(\eta_{n,j})\}_{j\ge 0}\qquad m\ge -1
\end{equation}

For each $m$, $S_m$ is a collection of submultiplicative semi-norms on
$\mathbb C[\Gamma_m]$.

\begin{definition}\label{def:cmsimp} For all $n\ge -1$, $H^{CM}_{L\hskip-.015in .}(\Gamma\hskip-.03in .^+)_n := {\cal
H}_{S_n}(\Gamma_n)$, the completion of $\mathbb C[\Gamma_n]$ with respect to the
collection of semi-norms $S_n$.
\end{definition}

We will write $\{[n]\mapsto H^{CM}_{L\hskip-.015in .}(\Gamma\hskip-.03in .^+)_n\}_{n\ge -1}$ simply as
$H^{CM}_{L\hskip-.015in .}(\Gamma\hskip-.03in .^+)$. 
\vskip.2in

Let $(c.\, complexes)$ denote the category of connected chain complexes over $\mathbb C$, with morphisms consisting of degree $0$ chain maps. The term \underbar{augmented simplicial chain complex} will refer to a covariant functor $F:(\Delta^+)^{op}\to (c.\, complexes)$, which we will write as $(F\hskip-.03in .^+)_* = \{[n]\mapsto (F_n)_*\}_{n\ge -1}$. For each $n\ge -1$, $(F_n)_* = ((F_n)_k,d_{n,k})_{k\ge 0}$ is a connected chain complex, with face and degeneracy maps corresponding to morphisms of complexes. We say that $(F\hskip-.03in .^+)_*$ is \underbar{of resolution type} if the augmented simplicial abelian group $\{[n]\mapsto (F_n)_k\}_{n\ge -1}$ is a resolution for each $k\ge 0$. Let $F_{**}$ denote the bicomplex associated to the simplicial complex $(F\hskip-.03in .)_* := \{[n]\mapsto (F_n)_*\}_{n\ge 0}$, and $Tot(F_{**})$ its total complex. Then the augmentation map $\varepsilon_*:(F\hskip-.03in .)_*\twoheadrightarrow (F_{-1})_*$ induces a chain map $Tot(F_{**})\twoheadrightarrow (F_{-1})_*$ which is a quasi-isomorphism (homology isomorphism) when $(F\hskip-.03in .^+)_*$ is of resolution type. Dually, a coaugmented cosimplicial (connected) cochain complex is of resolution type if each of the cosimplicial abelian groups $\{[n]\mapsto (F^n)^k\}_{n\ge -1}$ is a resolution for each $k\ge 0$. In all cases considered below, the dual situation arises via passage to either linear or continuous duals.
\vskip.2in

In conjunction with the procedure indicated in Definition \ref{def:cmsimp}, type $P$ resolutions of a discrete group $\pi$ can be used to construct continuous augmented simplicial resolutions of $H^{CM}_L(\pi)$ in the category of Fr\'echet algebras.

\begin{theorem}\label{thm:CMres} $H^{CM}_{L\hskip-.015in .}(\Gamma\hskip-.03in .^+)$ is an augmented
simplicial Fr\'echet algebra which in degree $-1$ is the Connes-Moscovici
algebra $H^{CM}_{L_{-1}}(\Gamma_{-1})$. When $(\Gamma\hskip-.03in .^+,L\hskip-.015in .^+)$ is a type $P$
resolution, the augmented simplicial complexes
\begin{gather*}
\{[n]\mapsto CH_*^t(H^{CM}_{L\hskip-.015in .}(\Gamma\hskip-.03in .^+)_n)\}_{n\ge -1}\\
\{[n]\mapsto CC_*^t(H^{CM}_{L\hskip-.015in .}(\Gamma\hskip-.03in .^+)_n)\}_{n\ge -1}
\end{gather*}

are of resolution type, as are the cosimplicial cocomplexes
\begin{gather*}
\{[n]\mapsto CH^*_t(H^{CM}_{L\hskip-.015in .}(\Gamma\hskip-.03in .^+)_n)\}_{n\ge -1}\\
\{[n]\mapsto CC^*_t(H^{CM}_{L\hskip-.015in .}(\Gamma\hskip-.03in .^+)_n)\}_{n\ge -1}
\end{gather*}
\end{theorem}

\begin{proof} By Prop.\ \ref{prop:map} all the face and degeneracy maps of the
augmented simplicial group algebra $\mathbb C[\Gamma\hskip-.03in .^+]$ extend to continuous
algebra homomorphisms satisfying the same augmented simplicial identies.
This makes $H^{CM}_{L\hskip-.015in .}(\Gamma\hskip-.03in .^+)$ an augmented simplicial algebra. Moreover,
for each $m\ge -1$, the set $S_m$ in (\ref{eqn:set}) above is countable. Thus the
completion in each degree is with respect to a countable collection of
semi-norms, hence Fr\'echet. We note the condition that the
word-length function takes values in $\mathbb N$ guarantees it is nice in the sense of Lemma \ref{lemma:bound}.
\vskip.2in

Assume now that $(\Gamma\hskip-.03in .^+,L\hskip-.015in .^+)$ is a type $P$ resolution. In order to show that the augmented simplicial (co-)complexes listed above are of resolution type, it will suffice to show that the augmented simplicial Fr\'echet space $H^{CM}_{L\hskip-.015in .}(\Gamma\hskip-.03in .^+)$ (gotton by forgetting the multiplication) admits a continuous simplicial contraction 
\[
\{\widetilde s_{n}:H^{CM}_{L\hskip-.015in .}(\Gamma\hskip-.03in .^+)_{n-1}\to H^{CM}_{L\hskip-.015in .}(\Gamma\hskip-.03in .^+)_n\}_{n\ge 0}
\]
 By assumption, $\Gamma(\varepsilon)\hskip-.015in .$ admits a 
simplicial contraction $s\hskip-.015in .' = \{s'_{n+1} : \Gamma(\varepsilon)_n \to
\Gamma(\varepsilon)_{n+1}\}_{n\ge 0}$ which is p-bounded. Fix a section
 $s(0):\Gamma_{-1}\rightarrowtail \Gamma_0,\, \varepsilon_0 \circ
s(0) =$ identity, with $s(0)(1) = 1$ and $s(0)$ minimal with respect to
word-length. Define $s(n) := s^{(n)}_0 \circ s(0): \Gamma_{-1} \rightarrowtail
\Gamma_n$.  Note that

\begin{equation}\label{eqn:identities}
\begin{gathered}
\varepsilon_n\circ s(n) = \text{ identity } \qquad \forall\, n\ge 0\; , \\
\partial_i \circ s(n) = s(n-1) \qquad\quad \forall n\ge 1, 0\le i \le n\;
, \\ s_i \circ s(n-1) = s(n) \qquad\qquad \forall n\ge 1, 0\le i \le
n-1\; . 
\end{gathered}
\end{equation}

For arbitrary $g\in \Gamma_n$, $g(s(n)(\varepsilon_n(g)))^{-1}\in
\Gamma(\varepsilon)_n$. We define $\widetilde s_{n+1}$ by
\begin{equation}\label{eqn:contraction}
\widetilde s_{n+1}(g) = s'_{n+1}(g(s(n)(\varepsilon_n(g)))^{-1})s(n+1)(\varepsilon_n(g))
\end{equation}

This defines a map of sets $\widetilde s_{n+1} :\Gamma_n\to \Gamma_{n+1}$, which
extends uniquely to a map of vector spaces $\widetilde s_{n+1} : \mathbb C[\Gamma_n]\to \mathbb C[\Gamma_{n+1}]$ where $\widetilde s_0 = s(0)$. By the simplicial identities
\begin{equation}\label{eqn:chaincontract}
d_{n+1}\widetilde s_{n+1} = (-1)^{n+1}(id) + \widetilde s_nd_n
\end{equation}
where $d_m = \sum_{i=0}^m \partial_i : \mathbb C[\Gamma_m]\to \mathbb C[\Gamma_{m-1}]$. It follows that $\widetilde s_* = \{\widetilde s_n\}_{n\ge 0}$ is a
p-bounded contraction of the augmented simplicial vector space $\mathbb
C[\Gamma\hskip-.03in .^+]$. Our goal will be to show that $\widetilde s_*$ is continuous with respect
to the semi-norms used to complete $\mathbb C[\Gamma\hskip-.03in .^+]$ to form
$H^{CM}_{L\hskip-.015in .}(\Gamma\hskip-.03in .^+)$. By hypothesis, $(\Gamma_n, L_n)$ is a free group with proper word-length metric when $n\ge 0$. However,
it may be infinitely generated. For a finitely generated free group $F$ on
more than one generator, one has Haagerup's inequality [H]

\begin{equation}\label{eqn:Haag}
\|\varphi\| \le 2\nu_{2,2}(\varphi)
\end{equation}

\begin{observation}\label{obs:inf}
Let $(F,\mathbb X,L)$ represent a free group $F$ on a countable
generating set $\mathbb X$, where $L$ is a word-length metric on $F$ induced by a proper function on the generating set $\mathbb X$ [O1]. Then (as in the case of finitely generated free groups) the inequality $\|x\| \le 2\nu_2(x)$ is satisfied for all $x\in H^{2,\infty}_L(F)$.
 Consequently, $H^{2,\infty}_L(F)$ is a dense, holomorphically closed involutive subalgebra of $C^*_r(F)$ containing $\mathbb C[F]$, and the completions of $\mathbb C[F]$ with respect to the collections of semi-norms $\{\nu_{2,s}\}_{s\ge 0}$ and $\{\eta_m\}_{m\ge 0}$ are the same.
\end{observation}

\begin{proof}
 Let $\mathbb X_m = \{x\in \mathbb X\;|\; L(x)\le m\}$, and let $F_m$ be the subgroup of $F$ generated by $\mathbb X_m$, 
where $L_m = L|_{F_m}$. We wish to show that $H^{\infty}_L(F)$ is contained in $C^*_r(F)$. Let 
$x = \sum_g\lambda(g)g\in H^{2,\infty}_L(F)$, and for each
$p$ let $x_p = \sum_{g\in F_p}\lambda(g)g$. For all $N > m,n$
\begin{equation}\label{eqn:normN}
\|x_m - x_n\|_N \le 2\nu_{2,2,N}(x_m - x_n)
\end{equation}
by (\ref{eqn:Haag}), where the norm on the left is the $C^*$ norm in $C^*_r(F_N)$, and the norm
on the right is in
$H^{2,\infty}_{L_N}(F_N)$. Denote the norm in $C^*_r(F)$ by $\|_-\|$.
Then (\ref{eqn:normN}) implies
\[
\|x_m - x_n\| \le 2\nu_{2,2}(x_m - x_n)
\]
Thus the Cauchy sequence $\{x_p\}$ converges in $C^*_r(F)$. Since it
converges to $x$ in $\ell^2(F)$, $x\in C^*_r(F)$.
\end{proof}
\vskip.2in

Returning to the proof at hand, we consider two cases.
\vskip.2in

\underbar{$n = 0$} We need to show $\widetilde s_0 = s(0) : \mathbb C[\Gamma_{-1}]\to
\mathbb C[\Gamma_0]$ is continuous with respect to the semi-norms in $S_0$. For each $m$, the set $S_{0,m}$ consists of
a single element $s_0^{(m)} : \Gamma_0\to \Gamma_m$. The proof for this case
follows from the sequence of inequalities
\begin{gather*}
(s_0^{(m)})^*(\eta_{m,j})(s(0)(x)) = \eta_{m,j}(s_0^{(m)}(s(0)(x)))\\
\le (2)2^j\nu_{2,m,j+2}(s_0^{(m)}(s(0)(x)))\qquad\qquad \text{by Lemma \ref{lemma:bound}
and Observation \ref{obs:inf}}\\
\le C_m 2^{j+1} \nu_{2,0,j+2+k_m}(s(0)(x))\qquad\qquad \text{as $s_0^{(m)}$
is p-bounded}\\
= C_m 2^{j+1}\nu_{2,-1,j+2+k_m}(x)\qquad\qquad \text{as $s(0)$ preserves
word-length}\\
\le C_m 2(j+2+k_m)2^{j+1}\eta_{-1,2(j+2+k_m)}(x)\qquad\qquad \text{by
Lemma \ref{lemma:bound}}
\end{gather*}
\vskip.2in

\underbar{$n > 0$} A semi-norm in $S_n$ is of the form $f^*(\eta_{m,j})$
where $f\in S_{n,m}$. Given $\widetilde s_{n} : \Gamma_{n-1}\to \Gamma_{n}$, the composition $f\circ \widetilde s_{n} : \Gamma_{n-1}\to \Gamma_m$
may be factored as
\[
\Gamma_{n-1}\overset {s_J\circ\partial_I}{\longrightarrow}
\Gamma_l\overset{\widetilde s_K}{\longrightarrow}\Gamma_m
\]
where $s_J$ is an iterated degeneracy, $\partial_I$ an iterated
composition of face maps, and $\widetilde s_K = \widetilde s_m\circ \widetilde
s_{m-1}\circ\dots\circ \widetilde s_{l+1}$ with $l < m$. This implies
\[
f^*(\eta_{m,j})(\widetilde s_{n}(x)) = \eta_{m,j}(\widetilde s_K((s_J\circ
\partial_I)(x)))
\]

If $m = -1$, then $f\circ \widetilde s_{n} : \Gamma_{n-1}\to \Gamma_{-1}$ is the
augmentation map $\varepsilon_{n-1}$, and
$$
\eta_{-1,j}(f(\widetilde s_{n}(x))) = \eta_{-1,j}(\varepsilon_n(x)) =
\varepsilon_{n-1}^*(\eta_{-1,j})(x)
$$

If $m\ge 0$, then
\begin{gather*}
\eta_{m,j}(\widetilde s_K((s_J\circ \partial_I)(x)))\\
\le 2^{j+1}\nu_{2,m,j+2}(\widetilde s_K((s_J\circ
\partial_I)(x)))\qquad\qquad\text{by Lemma \ref{lemma:bound}
and Observation \ref{obs:inf}}\\
\le 2^{j+1}(C_N)\nu_{2,l,j+2+N}((s_J\circ
\partial_I)(x))\qquad\qquad\text{by the p-boundedness of $\widetilde s_K$}\\
\le 2^{j+1}C_N(2(j+2+N))\eta_{l,2(j+2+N)}((s_J\circ \partial_I)(x))\\
= 2^{j+1}C_N(2(j+2+N))(s_J\circ \partial_I)^*(\eta_{l,2(j+2+N)})(x)
\end{gather*}
\vskip.1in

In other words, starting with the semi-norm $\eta = f^*(\eta_{m,j})\in
S_n$,
there is a constant $D_N$ and semi-norm
$\eta' = (s_J\circ \partial_I)^*(\eta_{l,2(j+2+N)})\in S_{n-1}$ such that
$$
\eta\circ \widetilde s_n\le D_N \eta'
$$
We conclude that for all $n\ge -1$ and semi-norms $\eta\in S_{n}$, there is a
semi-norm $\eta'\in S_{n-1}$ and a constant $C$ such that $\eta\circ \widetilde
s_{n} \le C\eta'$, verifying that for each $n\ge 0$, $\widetilde s_{n}:\mathbb C[\Gamma_{n-1}]\to \mathbb C[\Gamma_n]$ extends to a continuous morphism of Fr\'echet spaces
\begin{equation}
\widetilde s_{n}:H^{CM}_{L\hskip-.015in .}(\Gamma\hskip-.03in .^+)_{n-1}\to H^{CM}_{L\hskip-.015in .}(\Gamma\hskip-.03in .^+)_n
\end{equation}
Because $\widetilde s_*$ is a simplicial contraction on the dense augmented simplicial subspace $\mathbb C[\Gamma\hskip-.03in .^+]$, it follows by continuity that its extension satisfies the necessary identities to make it a simplicial contraction of the augmented simplicial Fr\'echet space $H^{CM}_{L\hskip-.015in .}(\Gamma\hskip-.03in .^+)$, completing the proof.
\end{proof}
\vskip.2in

One also has the $\ell^1$-rapid decay and algebraic versions of the previous theorem; the proof in both cases is quite straightforward, and is left to the reader.

\begin{theorem}\label{thm:ell1res} For an augmented simplicial group $(\Gamma\hskip-.03in .^+,L\hskip-.015in .^+)$, let $A(\Gamma\hskip-.03in .^+)$ denote either the augmented simplicial group algebra $\mathbb C[ \Gamma\hskip-.03in .^+]$ (equipped degreewise with the fine topology), or the augmented simplicial Fr\'echet algebra $H^{1,\infty}_{L\hskip-.015in .}(\Gamma\hskip-.03in .^+)$ (equipped degreewise with the Fr\'echet topology). If $(\Gamma\hskip-.03in .^+,L\hskip-.015in .^+)$ is a type $P$
resolution, the augmented simplicial complexes
\begin{gather*}
\{[n]\mapsto CH_*^t(A(\Gamma_n))\}_{n\ge -1}\\
\{[n]\mapsto CC_*^t(A(\Gamma_n))\}_{n\ge -1}
\end{gather*}

are of resolution type, as are the cosimplicial cocomplexes
\begin{gather*}
\{[n]\mapsto CH^*_t(A(\Gamma_n))\}_{n\ge -1}\\
\{[n]\mapsto CC^*_t(A(\Gamma_n))\}_{n\ge -1}
\end{gather*}
\end{theorem}
\vskip.2in

The definition of $H^{CM}_{L\hskip-.015in .}(\Gamma\hskip-.03in .^+)$ given above involves the minimum amount of adjustment necessary to extend the construction of the Connes-Moscovici algebra $H^{CM}_L(\pi)$ to a type $P$ resolution of $\pi$. It does not allow for arbitrary homomorphisms of augmented simplicial groups $\Gamma\hskip-.03in .^+\to \Gamma'\hskip-.03in .^+$ to be extended to $H^{CM}_{L\hskip-.015in .}(\Gamma\hskip-.03in .^+)$. Over the category $(f.p.groups)$, this can be rectified as follows.
\vskip.2in

For each object $G$ of $(f.p.groups)$, let
\[
T(G) := \underset{G'\in obj(f.p.groups)}{\bigcup}Hom(G,G')
\]
be the set of all morphisms in $(f.p.groups)$ originating with $G$. For each $f\in T(G), f:G\to G'$, let $\ell^2(f)$ denote a copy of $\ell^2(G')$ indexed by $f$. Let $\rho_f:G\to {\cal L}(\ell^2(f))$ denote the unitary representation induced by the left regular representation of $G'$ on $\ell^2(f) = \ell^2(G')$ precomposed with the homomorphism $f$. Let $\ell^2(T(G)) := \oplus_{f\in T(G)} \ell^2(f)$ be the direct sum of the Hilbert spaces $\ell^2(f)$, and let $\rho_{T(G)}:= \oplus_{f\in T(G)}\rho_f$ be the direct sum of the unitary representations $\rho_f$. Then $C_m^*(G)$ will denote the completion of $\mathbb C[G]$ in the operator norm associated to the unitary representation $(\rho_{T(G)},\ell^2(T(G)))$. This $C^*$-algebra lies between the maximal $C^*$-algebra $C^*(G)$ and the reduced $C^*$-algebra $C^*_r(G)$, in that there is a factorization
\[
C^*(G)\to C^*_m(G)\to C^*_r(G)
\]
for each finitely presented group $G$. Moreover, the map $C^*(G)\to C^*_m(G)$ is induced by a natural transformation of functors $C^*(_-)\to C^*_m(_-):(f.p.groups)\to (C^*$-$algebras)$.
\vskip.2in

We can now duplicate the Connes-Moscovici construction on this sum. For each $f\in T(G)$, define
$D^f_L : \ell^2(T(G))\to \ell^2(T(G))$ on basis elements by 
\begin{itemize}
\item $D^f_L(\delta_g) = \delta_g$ if $\delta_g\in\ell^2(f')$, $f'\ne f$
\item $D^f_L(\delta_g) = L_{st}(g)\delta_g$ when $\delta_g\in\ell^2(f)$. 
\end{itemize}

As before, each operator $D^f_L$  defines an unbounded operator $\partial^f_L : {\cal L}(\ell^2(T(G)))\to {\cal L}(\ell^2(T(G)))$ given by
\[
\partial^f_L(M) = ad(D^f_L)(M) = D^f_L\circ M - M\circ D^f_L
\]

We can now define the  \underbar{maximal Connes-Moscovici algebra} associated to $G$ as

\begin{definition}\label{def:maximal} $H^{CM}_m(G) = \{a\in C^*_m(G)\,|\, (\partial^f_L)^k(a)\in {\cal L}(\ell^2(T(G)))\,\text{ for all } k\ge 0, f\in T(G)\}$.
\end{definition}

This algebra admits semi-norms induced by the $(\partial^f_L)^k$:

\begin{equation}\label{eqn:etamm}
\eta^f_n(a) := \sum_{i=0}^n \binom ni \|(\partial^f_L)^i(a)\|_m
\end{equation}

where $\|_-\|_m$ is the $C^*$-norm on $C^*_m(G)$ and $\partial_L^0 = Id$.
Exactly as above, one has for each $f\in T(G)$ inequalities

\begin{gather*}
\eta^f_n(ab)\le \eta^f_n(a)\eta^f_n(b)\\
\eta^f_m(a)\le \eta^f_n(a)\quad\text{when}\, m\le n
\end{gather*}
\vskip.2in

\begin{remark} {\rm In Definition \ref{def:cmsimp} above, we could have proceeded as follows: let $T_m = \coprod_{n\ge -1}S_{m,n}$ denote the set of maps in the simplicial group $\Gamma\hskip-.03in .$ originating with $\Gamma_m$. For each $T_m\ni f:\Gamma_m\to \Gamma_n$, let $\ell^2(f)$ denote a copy of $\ell^2(\Gamma_n)$ indexed by $f$. Write $\ell^2(T_m)$ for $\oplus_{f\in T_m}\ell^2(f)$, $\rho(T_m)$ for $\oplus_{f\in T_m}\rho(f)$, where $\rho(f)$ is defined as above. For each $m\ge 0$, $C^*_{T\hskip-.03in .}(\Gamma\hskip-.03in .)$ in degree $-1$ is $C^*_r(\pi)$, $\pi = \pi_0(\Gamma\hskip-.03in .)$, and in dimension $n\ge 0$ is the $C^*$-algebra completion of $\mathbb C[\Gamma_n]$ with respect to the unitary representation $(\rho(T_n),\ell^2(T_n))$. This is a simplicial $C^*$-algebra with $\pi_0(C^*_{T\hskip-.03in .}(\Gamma\hskip-.03in .))$ mapping to $C^*_r(\pi)$ by a map inducing an isomorphism in topological $K$-theory. Similarly, for $n=-1$ set $H^{CM}_{T\hskip-.03in .}(\Gamma\hskip-.03in .) = H^{CM}_L(\pi)$, and for $n\ge 0$
\[
H^{CM}_{T\hskip-.03in .}(\Gamma\hskip-.03in .)_n := \{a\in C^*_{T\hskip-.03in .}(\Gamma\hskip-.03in .)_n\,|\, (\partial^f_L)^k(a)\in {\cal L}(\ell^2(T_n))\,\text{ for all } k\ge 0, f\in T_n\}
\]
The fact that $T_n$ is a finite set of group homomorphisms for each $n$ implies that the natural map
$H^{CM}_{T\hskip-.03in .}(\Gamma\hskip-.03in .)_n\to H^{CM}_{L\hskip-.015in .}(\Gamma\hskip-.03in .)_n$ is an isomorphism for each $n$. In the more general case, when the set of homomorphisms is infinite, these two constructions are not equivalent, and the proper construction is that given in Def. (\ref{def:maximal}).}
\end{remark}
\vskip.2in

\begin{proposition}\label{prop:cmmax} The association $G\mapsto H^{CM}_m(G)$ defines a functor
\[
H^{CM}_m(_-): (f.p.groups)\to (Fr\acute{e}chet\ algebras)
\]
\end{proposition}

\begin{proof} As we have seen, using the standard word-length function on the objects of $(f.p.groups)$ allows one to enumerate both the objects of $(f.p.groups)$ and, for any pair of objects $G,G'\in obj(f.p.groups)$, the Hom set $Hom(G,G')$. Consequently $T(G)$ is countable for each $G\in obj(f.p.groups)$. The functoriality is built in to the definition of the algebra (compare Prop. \ref{prop:map} above).
\end{proof}

The functoriality of both $H^{CM}_m(_-)$ and $C^*_m(_-)$ allows one to treat the inclusion $H^{CM}_m(G)\hookrightarrow C^*_m(G)$ as the restriction to $G\in obj(f.p.groups)$ of a natural transformation
\[
H^{CM}_m(_-)\hookrightarrow C^*_m(_-)
\]
If $\Gamma\hskip-.03in .$ is a simplicial object in $(f.p.groups)$, then this natural transformation induces a continuous homomorphism of simplicial algebras
\[
H^{CM}_m(\Gamma\hskip-.03in .)\hookrightarrow C^*_m(\Gamma\hskip-.03in .)
\]

For the following proposition, we recall that $A^f$ indicates the continuous topological algebra $A$ is equipped with the fine topology.

\begin{proposition} Let $\Gamma\hskip-.03in .$ be a simplicial object in $(f.p.groups)$, with $\pi := \pi_0(\Gamma\hskip-.03in .)$. Then there is an isomorphism
\[
K_*^t(\pi_0(H^{CM}_m(\Gamma\hskip-.03in .)^f))\cong K_*^t(C^*_m(\pi)),\quad *\ge 1
\]
Moreover, the natural map $\mathbb C[\pi] \to C^*_m(\pi)$ factors as
\begin{equation}\label{eqn:factor}
\mathbb C[\pi] = \pi_0(\mathbb C[\Gamma\hskip-.03in .])\to \pi_0(H^{CM}_m(\Gamma\hskip-.03in .)^f)\to C^*_m(\pi)
\end{equation}
\end{proposition}

\begin{proof} 
For a simplicial algebra $A\hskip-.015in .$, set $A^j_n := \cap_{i = 0}^j ker(\partial_i:A_n\to A_{n-1})$ for $0\le j\le n$ and $n\ge 1$. Also, set $A^0_0 := \partial_1(A^0_1)$. Then $\pi_0(A\hskip-.015in .) = A_0/A^0_0$. As in [O2, Cor. 1.7], one identifies $C^*_m(\Gamma\hskip-.03in .)^0_0$ as the norm-closure of $\mathbb C[\Gamma\hskip-.03in .]^0_0$ in $C^*_m(\Gamma_0)$, with $\pi_0(C^*_m(\Gamma\hskip-.03in .)) = C^*_m(\pi)$. Following [O2, Lemma 5.2] and applying Baum retopologization [Appendix], one has that the map $H^{CM}_m(\Gamma\hskip-.03in .) ^j_n\to C^*_m(\Gamma\hskip-.03in .)^j_n$ induces an isomorphism on topological $K$-groups
\begin{equation}\label{eqn:inviso}
K_*^t((H^{CM}_m(\Gamma\hskip-.03in .) ^j_n)^f)\cong K_*^t(C^*_m(\Gamma\hskip-.03in .)^j_n),\quad *\ge 0
\end{equation}
for all $0\le j\le n$, $n\ge 0$. The isomorphism in (\ref{eqn:inviso}) induces an isomorphism\footnote{ see Appendix}
\[
K_*^t(\pi_0(H^{CM}_m(\Gamma\hskip-.03in .)^f))\cong K_*^t(\pi_0(C^*_m(\Gamma\hskip-.03in .))) = K_*^t(C^*_m(\pi)),\quad *\ge 1
\]
To verify the factorization of (\ref{eqn:factor}) above, it suffices to observe the existence of a commutative diagram
\vskip.2in
\centerline{
\xymatrix{
\mathbb C[\Gamma_1]\ar[d]^{\partial_i}\ar[r] & H^{CM}_m(\Gamma_1)\ar[d]^{\partial_i}\\
\mathbb C[\Gamma_0]\ar[r] & H^{CM}_m(\Gamma_0)
}
}
\vskip.2in
which, in turn, follows from the functoriality of both $C^i(_-)$ and $H^{CM}_m(_-)$ on the category $(f.p.groups)$ once one knows that for each finitely presented group $G$, the natural inclusion $\mathbb C[G]\hookrightarrow C^*_m(G)$ factors as
\[
\mathbb C[G]\hookrightarrow H^{CM}_m(G)\to C^*_m(G)
\]
\end{proof}
\newpage


\section{Detecting the assembly map}
\vskip.4in

\subsection{Spectral sequences in Hochschild and cyclic (co-)homology}
\vskip.3in

 We begin with the observation (compare[O1, Appendix]) that if $\Gamma\hskip-.03in .$ is a free simplicial resolution of $\pi$, the homology and cohomology of $B\pi$ (with coefficients) may be computed as follows: let $V$ be a trivial $\pi$-module. Set $D_n := H_1(B\Gamma_{n-1};\mathbb Z)$, with $d_n = \sum_{i=0}^{n-1}(-1)^i (\partial_i)_*:D_n\to D_{n-1}$, $\partial_i:\Gamma_n\to \Gamma_{n-1}$. Then $(D_*,d_*)$ is a chain complex; its dual (with coefficients in $V$) is given by $D^n(V) := H^1(B\Gamma_{n-1};V), \delta^n = \sum_{i=0}^n (-1)^i (\partial_i)^*:D^n\to D^{n+1}$. For $n\ge 1$ one then has
\begin{equation}
\begin{gathered}\label{eqn:isom}
H_n(B\pi;V) = H_n(D_*\otimes V,d_*\otimes 1)\\
H^n(B\pi;V) = H^n(D^*(V))
\end{gathered}
\end{equation}
Next, consider the spectral sequence in cyclic homology
\begin{equation}\label{eqn:spec-seq-alg}
E^1_{**} = \{E^1_{p,q} = HC_p(\mathbb C[\Gamma_q]) \Rightarrow HC_{p+q}(\mathbb C[\Gamma\hskip-.03in .]) \cong HC_{p+q}(\mathbb C[\pi])\}_{p,q\ge 0}
\end{equation}
When $p > 0$ we have $HC_p(\mathbb C[\Gamma_q]) = HC_p(\mathbb C[\Gamma_q])_{<1>}$ for all $q\ge 0$. For any free group $F$, the cyclic homology of $\mathbb[F]$ above dimension $0$ is given by
\begin{equation}\label{eqn:cyclic-free}
\begin{gathered}
HC_{2p+1}(\mathbb C[F]) = H_1(BF),\quad p\ge 0\\
HC_{2p}(\mathbb C[F]) = \mathbb C,\quad p\ge 1
\end{gathered}
\end{equation}
Graphically, then, the $E^1_{**}$-term in (\ref{eqn:spec-seq-alg}) may be represented as
\begin{equation}\label{diagram1}
\xymatrix{
\vdots \dto & \vdots \dto & \vdots \dto & \vdots \dto \\
HC_0(\mathbb C[\Gamma_3])\dto & HC_1(\mathbb C[\Gamma_3])\dto & HC_2(\mathbb C[\Gamma_3])\dto &
HC_3(\mathbb C[\Gamma_3])\dto &\cdots\\
HC_0(\mathbb C[\Gamma_2])\dto & HC_1(\mathbb C[\Gamma_2])\dto & HC_2(\mathbb C[\Gamma_2])\dto &
HC_3(\mathbb C[\Gamma_2])\dto &\cdots\\ 
HC_0(\mathbb C[\Gamma_1])\dto & HC_1(\mathbb C[\Gamma_1])\dto & HC_2(\mathbb C[\Gamma_1])\dto &HC_3(\mathbb C[\Gamma_1])\dto & \cdots\\ 
HC_0(\mathbb C[\Gamma_0]) & HC_1(\mathbb C[\Gamma_0]) & HC_2(\mathbb C[\Gamma_0]) & HC_3(\mathbb C[\Gamma_0]) & \cdots
}
\end{equation}
Incorporating the isomorphisms of (\ref{eqn:cyclic-free}), the diagram in (\ref{diagram1}) becomes
\begin{equation}\label{diagram2}
\xymatrix{
\vdots \dto & \vdots \dto & \vdots \dto^{id} & \vdots \dto & \vdots \dto^{id} & \vdots \dto \\
HC_0(\mathbb C[\Gamma_3])\dto & H_1(B\Gamma_3)\dto & \mathbb C\dto^{0} &
H_1(B\Gamma_3)\dto & \mathbb C \dto^{0} & H_1(B\Gamma_3)\dto &\cdots\\
HC_0(\mathbb C[\Gamma_2])\dto & H_1(B\Gamma_2)\dto & \mathbb C\dto^{id} &
H_1(B\Gamma_2)\dto & \mathbb C \dto^{id} & H_1(B\Gamma_2)\dto &\cdots\\
HC_0(\mathbb C[\Gamma_1])\dto & H_1(B\Gamma_1)\dto & \mathbb C\dto^{0} &
H_1(B\Gamma_1)\dto & \mathbb C\dto^{0} & H_1(B\Gamma_1)\dto &\cdots\\
HC_0(\mathbb C[\Gamma_0]) & H_1(B\Gamma_0) & \mathbb C & H_1(\Gamma_0) & \mathbb C & H_1(\Gamma_0) & \cdots
}
\end{equation}
\vskip.2in

Because $HC_*(\mathbb C[\pi])_{<1>} = \oplus_{m\ge 0} H_{*-2m}(B\pi;\mathbb C)$ is a summand of $HC_*(\mathbb C[\pi])$ to which the above spectral sequence is converging, the isomorphisms of (\ref{eqn:isom}) imply
\begin{itemize}
\item The spectral sequence above, with $E^1_{**}$-term given as in (\ref{diagram1}) or (\ref{diagram2}), collapses at the $E^2_{**}$-term; i.e., $E^2_{**} = E^{\infty}_{**}$.
\item For each $m\ge 0$, the summand $H_{*-2m}(B\pi;\mathbb C)$ of $HC_*(\mathbb C[\pi])_{<1>}$ identifies with $E^2_{2m+1,*-2m-1}$ when $(*-2m) > 0$. In particular, the unshifted summand $H_*(B\pi;\mathbb C)$ of $HC_*(\mathbb C[\pi])_{<1>}$ appears as $E^2_{1,*-1} = H_*(D_*,d_*)$.
\end{itemize}
\vskip.2in

The same analysis applies to the spectral sequence in cyclic cohomology. The relevant points are
\begin{itemize}
\item The spectral sequence 
\begin{equation}
E_1^{**} = \{E_1^{pq} := HC^p(\mathbb C[\Gamma_q])\Rightarrow HC^{p+q}([\Gamma\hskip-.03in .])\cong HC^{p+q}(\mathbb C[\pi])\}_{p,q\ge 0}
\end{equation}
collapses at the $E_2^{**}$-term.
\item The summand $H^{*-2m}(B\pi;\mathbb C)$ of $HC^*(\mathbb C[\pi])_{<1>}$ identifies with the $E_2^{2m+1,*-1}$-term of this spectral sequence.
\end{itemize}
\vskip.2in

Note also that analogous spectral sequences exist in Hochschild homology and cohomology:
\begin{gather*}
E^1_{**} = \{E^1_{pq} := HH_p(\mathbb C[\Gamma_q])\Rightarrow HH_{p+q}([\Gamma\hskip-.03in .])\cong HH_{p+q}(\mathbb C[\pi])\}_{p,q\ge 0}\\
E_1^{**} = \{E_1^{pq} := HH^p(\mathbb C[\Gamma_q])\Rightarrow HH^{p+q}([\Gamma\hskip-.03in .])\cong HH^{p+q}(\mathbb C[\pi])\}_{p,q\ge 0}
\end{gather*}
and that these spectral sequences in both Hochschild and cyclic (co)homology exist for arbitrary simplicial groups $\Gamma\hskip-.03in$. Moreover, when $\Gamma\hskip-.03in .$ is degreewise free, (but not necessarily a resolution of $\pi$), the description of the $E^2_{**}$-term in cyclic homology given above in (\ref{diagram2}) still applies. In this case, however, the spectral sequence need not collapse at the $E^2_{**}$-term, and will converge to $HC_*(\mathbb C[\Gamma\hskip-.03in .])$ which may not equal $HC_*(\mathbb C[\pi])$.
\vskip.2in

From these considerations, one sees that an arbitrary $x\in H_n(B\pi;\mathbb C)$ admits a unique representative as 
\begin{equation}\label{eqn:xrep}
x\in E^2_{1,n-1}= ker\left(HC_1(\mathbb C[\Gamma_{n-1}])\to HC_1(\mathbb C[\Gamma_{n-2}])\right)/im\left(HC_1(\mathbb C[\Gamma_{n}])\to HC_1(\mathbb C[\Gamma_{n-1}])\right)
\end{equation}

Suppose now that $\Gamma\hskip-.03in .$ is a type $P$ resolution of a finitely-presented group $\pi$. By Theorems \ref{thm:ell1res} and \ref{thm:CMres} above, one has corresponding spectral sequences in topological cyclic homology:
\begin{gather}\label{eqn:typePcyclic}
E^1_{**} = \{E^1_{p,q} = HC_p^t(H^{1,\infty}_{L_q}(\Gamma_q)) \Rightarrow HC_{p+q}^t(H^{1,\infty}_{L\hskip-.015in .}(\Gamma\hskip-.03in .)) \cong HC_{p+q}^t(H^{1,\infty}_L(\pi))\}_{p,q\ge 0}
\\
E^1_{**} = \{E^1_{p,q} = HC_p^t(H^{CM}_{L\hskip-.015in}(\Gamma\hskip-.03in .)) \Rightarrow HC_{p+q}^t(H^{CM}_{L\hskip-.015in .}(\Gamma\hskip-.03in .)) \cong HC_{p+q}^t(H^{CM}_L(\pi))\}_{p,q\ge 0}
\end{gather}
Given an element $0\ne x$ as in (\ref{eqn:xrep}) above, one can choose a representative $x\in H_1(B\Gamma_{n-1};\mathbb C)\cong HC_1(\mathbb C[\Gamma_{n-1}])$, and attempt to show that
\begin{itemize}
\item its image in $HC_1^t(H^{1,\infty}_{L_{n-1}}(\Gamma_{n-1}))$ survives to $E^{\infty}_{1,n-1}$, or even more, that
\item its image in $HC_1^t(H^{CM}_{L\hskip-.015in .}(\Gamma\hskip-.03in .)_{n-1})$ survives to $E^{\infty}_{1,n-1}$.
\end{itemize}

As we shall see, verifying the former for all (rational) homology classes implies rational injectivity of the assembly map for $K_*^t(\ell^1(\pi))$, while verifying the latter for all such classes implies rational injectivity of the assembly map for $K_*^t(C^*_r(\pi))$. However, even though these spectral sequences exist (and are a prime motivation for what follows), working with them directly to verify injectivity for homology classes for which injectivity was not already known is problematic. The first complicating factor is the size of the group $\Gamma\hskip-.03in .$, which affects both spectral sequences. An additional issue affecting the second spectral sequence is its lack of functoriality with repsect to simplicial group homomorphisms, as we have discussed above. In the next section, we introduce a technique which gets around the first problem, culminating in the proof of Theorem A. And, with $H^{CM}_{m}(\Gamma\hskip-.03in .)$ in place of $H^{CM}_{L\hskip-.015in .}(\Gamma\hskip-.03in .)$, it yields as well a reduction of the Strong Novikov Conejcture to a conjecture about the topological cyclic homology group $\ov{HC}_1^t(H^{CM}_m(F))$, where $F$ is a finitely-generated free group.
\vskip.5in


\subsection{The local Chern character associated to an integral homology class}
\vskip.3in

Suppose now that $\pi$ is finitely presented. In this case we can construct an augmented simplicial resolution $\Gamma\hskip-.03in .^+$ of $\pi$ of type $HF^2$; in other words, a p-bounded resolution degreewise free above dimension $-1$, with $\pi_0(\Gamma\hskip-.03in .) = \Gamma_{-1} = \pi$ and $\Gamma_i$ finitely generated for $i = 0,1$. As is true for any resolution, $H_*(B\pi;\mathbb C) = H_*(D_*\otimes\mathbb C,d_*)$ for $*\ge 1$, where $D_n = H_1(B\Gamma_{n-1};\mathbb Z)$, with differential induced by the alternating sum of the boundary maps in homology. This complex can be defined for any simplicial group (not just free resolutons); for that reason when we want to specify the dependence on the group $\Gamma\hskip-.03in .$, we will write $D_*(\Gamma\hskip-.03in .)$ rather than just $D_*$. As in [O1], we assume $\Gamma\hskip-.03in .$ comes equipped with a graded generating set $\mathbb X\hskip-.015in .$ which is closed under degeneracies, and which (via the finite presentability of $\pi$) we may arrange to be finite in dimensons $0$ and $1$.
\vskip.2in

In what follows, an {\it integral} homology class $x\in H_n(B\pi;\mathbb C)$, $n\ge 1$, will refer to an element in the image of the map $H_n(B\pi;\mathbb Z)\to H_n(B\pi;\mathbb C)$ induced by the inclusion of coefficients $\mathbb Z\hookrightarrow \mathbb C$. Via the identification $H_*(B\pi;\mathbb Z)\cong H_n(D_*(\Gamma\hskip-.03in .),d_*)$, we may choose a preimage $\tilde{x}\in\Gamma_{n-1}$ of $x\in H_n(D_*(\Gamma\hskip-.03in .),d_*)$ mapping to a cycle representing $x$ under the projection $\Gamma_{n-1}\twoheadrightarrow H_1(B\Gamma_{n-1};\mathbb Z)\to H_1(B\Gamma_{n-1};\mathbb C)$; from this element we will construct an augmented simplicial subgroup $\Gamma(\tilde{x})\hskip-.015in .^+\subseteq \Gamma\hskip-.03in .^+$. We begin by defining a graded set $S(\tilde{x})\hskip-.015in .$
\begin{itemize}
\item For $j = 0,1$, $S(\tilde{x})_j = \mathbb X_j$;
\item If $n\ge 3$, $S(\tilde{x})_{n-1}$ contains a single element $s_{\tilde{x}}$ corresponding to $\tilde{x}$;
\item $S(\tilde{x})_{n-2}$ contains every element of $\mathbb X_{n-2}$ used in writing $\prod_{0\le i\le n-1}\left(\partial_i(\tilde{x})\right)^{(-1)^i}$ as a product of commutators in $\Gamma_{n-2}$;
\item For all $2 \le j < n-2$, $S(\tilde{x})_j$ is the minimal subset of $\mathbb X_j$ containing all generators occuring in the unique reduced word representing $\partial_I(y)$, where $\partial_I(_-)$ denotes an iterated face map beginning either in dimension $n-1$ or $n-2$ and $y\in S(\tilde{x})_{n-1}\coprod S(\tilde{x})_{n-2}$.
\end{itemize}

Now define $\mathbb X(\tilde{x})\hskip-.015in .$ to be the closure of $S(\tilde{x})\hskip-.015in .$ under degeneracies, and $\Gamma(\tilde{x})\hskip-.015in .$ to be the free simplicial group generated by $\mathbb X(\tilde{x})\hskip-.015in .$ There is a canonical injection
\[
\Gamma(\tilde{x})\hskip-.015in .\hookrightarrow \Gamma\hskip-.03in .
\]
which sends $s_{\tilde{x}}$ to $\tilde{x}$, and on other generators is the identity map. In fact, one can define the simplicial structure on
$\Gamma(\tilde{x})\hskip-.015in .$ to be that making this inclusion of graded sets a simplicial map. We extend $\Gamma(\tilde{x})\hskip-.015in .$ to an augmented simplicial object by setting $\Gamma(\tilde{x})_{-1} = \pi$. The following proposition lists the essential properties of $\Gamma(\tilde{x})\hskip-.015in .$

\begin{proposition}\label{prop:localprops} For any element $\tilde{x}$,
\begin{itemize}
\item $\Gamma(\tilde{x})_i = \Gamma_i$ for $-1\le i\le 1$;
\item $\Gamma(\tilde{x})\hskip-.015in .$ is free and finitely generated in each degree;
\item For all $n\ge 1$ there is an integral \lq\lq fundamental\rq\rq homology class
$\mu_{\tilde{x}}\in H_n(B\Gamma(\tilde{x})\hskip-.015in .;\mathbb Z)$.
Moreover, we have $\mu_{\tilde{x}}\mapsto x$ under the composition in homology
\[
H_n(B\Gamma(\tilde{x})\hskip-.015in .;\mathbb Z)\to H_n(B\Gamma(\tilde{x})\hskip-.015in .;\mathbb C)\to H_n(B\Gamma\hskip-.03in .;\mathbb C)\cong H_n(B\pi;\mathbb C)
\]
If $n\ge 3$, $H_n(B\Gamma(\tilde{x})\hskip-.015in .;\mathbb Z)\cong \mathbb Z$, generated by $\mu_{\tilde{x}}$.
\item In the spectral sequences 
\begin{gather*}
E^1_{pq} = HC_p(\mathbb C[\Gamma(\tilde{x})_q])\Rightarrow HC_{p+q}(\mathbb C[\Gamma(\tilde{x})\hskip-.015in .])\quad\text{resp\ .}\\
E_1^{pq} = HC^p(\mathbb C[\Gamma(\tilde{x})_q])\Rightarrow HC^{p+q}(\mathbb C[\Gamma(\tilde{x})\hskip-.015in .])
\end{gather*}
the element $\mu_{\tilde{x}}$ resp.\ its dual is represented by a permanent cycle in $E^{\infty}_{1,n-1}$ resp.\ $E_{\infty}^{1,n-1}$ which is non-zero if $0\ne x\in H_n(B\pi)$.
\end{itemize}
\end{proposition}

\begin{proof}
The first two properties follow directly from the construction of $\Gamma(\tilde{x})\hskip-.015in .^+$. For the third, we observe that the image of $s_{\wt{x}}\in\Gamma(\wt{x})_{n-1}$ in $H_1(B\Gamma(\wt{x})_{n-1};\mathbb Z)$ is a cycle\footnote{The third defining property of the set $S(\tilde{x})$ guarantees that $[s_{\tilde{x}}]\in D_n(\Gamma(\tilde{x})\hskip-.015in .)$ is a cycle.} representing a canonically defined class $\mu_{\wt{x}}\in H_n(D(\Gamma(\wt{x})\hskip-.015in .;\mathbb Z)$. Now assume $n\ge 3$. The homology of $D_*(\Gamma(\tilde{x})\hskip-.015in .)$ can be computed by the normalized complex $\overline{D}_*(\Gamma(\tilde{x})\hskip-.015in .)$, in which degenerate elements have been identified with zero. In this normalized complex, $\overline{D}_m(\Gamma(\tilde{x})\hskip-.015in .)=0$ when $m > n$, $\overline{D}_n(\Gamma(\tilde{x})\hskip-.015in .)\cong \mathbb Z$ generated by the cycle $\mu_{\tilde{x}} := [s_{\tilde{x}}]$, and for $m < n$ is the free $\mathbb Z$-module with canonical basis consisting of the non-degenerate elements of $\mathbb X(\tilde{x})_m$. This identifies $H_n(D_*(\Gamma(\tilde{x})\hskip-.015in .);\mathbb Z)$ with $\mathbb Z$ on basis element $\mu_{\tilde{x}}$, which by construction maps to $x$ under the indicated map. The same result holds in cohomology, with an identification $H^n(D_*(\Gamma(\tilde{x})\hskip-.015in .);\mathbb Z)\cong\mathbb Z$ generated by the dual class $\mu_{\tilde{x}}^*$. We now consider the fourth claim. Because $\Gamma(\tilde{x})\hskip-.015in .$ is a free simplicial group, the cyclic homology spectral sequence converging to $HC_*(\mathbb C[\Gamma(\tilde{x})\hskip-.015in .])$ satisfies $E^2_{1,n-1}\cong H_n(D_*(\Gamma(\tilde{x})\hskip-.015in .)\otimes\mathbb C)$. By the third property, the inclusion $\Gamma(\tilde{x})\hskip-.015in .\hookrightarrow \Gamma\hskip-.03in .$ induces a map of cyclic homology spectral sequences which is an injection on the $E^2_{1,n-1}$-term. As $\Gamma\hskip-.03in .$ is both degreewise free and a resolution of $\pi$, the image of $\mu_{\tilde{x}}$ in the $E^2_{1,n-1}$-term for $HC_*(\mathbb C[\Gamma\hskip-.03in .])$ survives to $E^{\infty}_{1,n-1}$. But this implies it survives to $E^{\infty}_{1,n-1}$ in the cyclic homology spectral sequence converging to $HC_*(\mathbb C[\Gamma(\tilde{x})\hskip-.015in .])$. The dual argument applies in cyclic cohomology spectral sequence converging to $HC^*(\mathbb C[\Gamma(\tilde{x})\hskip-.015in .])$.
\end{proof}

In general, $\Gamma(\tilde{x})\hskip-.015in .$ will not be a resolution of $\pi = \pi_0(\Gamma(\tilde{x})\hskip-.015in .)$, let alone a type $P$ resolution, and so the simplicial projections $\mathbb C[\Gamma(\tilde{x})\hskip-.015in .]\surj \mathbb C[\pi]$, $H^{1,\infty}_{L\hskip-.015in .}(\Gamma(\tilde{x})\hskip-.015in .)\surj H^{1,\infty}_L(\pi)$ will not induce isomorphisms in (topological) cyclic homology. The following theorem allows us to bypass this issue as far as the construction of Chern characters are concerned.

\begin{theorem}\label{thm:factor} Suppose $A\hskip-.015in .^+$ is an augmented simplicial Fr\'echet algebra with $\pi_0(A\hskip-.015in .) = A_{-1}$, for which $A_{-1}$ is endowed with a (coarser) norm topology (in which it need not be complete). Then for $*\ge 1$ there is a factorization
\vskip.2in
\centerline{
\xymatrix{
& HC_*^t(A\hskip-.015in .)\ar[d]^{p_*}\\
K_*^t(\pi_0(A\hskip-.015in .))\ar@{-->}[ru]^{\widetilde{ch}_*^{CK}}\ar[r]^{ch_*^{CKT}} & HC_*^t(\pi_0(A\hskip-.015in .)^f)
}}
\vskip.2in
where $ch_*^{CKT}$ denotes the Connes-Karoub-Tillmann Chern character [constructed in Appendix]. Moreover, this factorization is functorial in $A$.
\end{theorem}

\begin{proof} Define
\[
\underline{HPer}_*^x(A.) := \underset{S}{\varprojlim} HC_{*+2n}^x(A.),\quad x = a,t
\]
The lifting indicated in the above diagram is given by the following composition
\vskip.2in
\centerline{
\xymatrix{
K_*^t(\pi_0(A\hskip-.015in .))\ar@{-->}[dd]^{\widetilde{ch}_*^{CK}} &  K_*^t(\pi_0(A\hskip-.015in .)^f)\ar[l]^{\cong}\ar[d]^{\widetilde{ch}_*^T} & \\
& \underline{HPer}_*^t(\pi_0(A\hskip-.015in .)^f)\ar@{=}[r] & \underline{HPer}_*^a(\pi_0(A\hskip-.015in .)) \\
 HC_*^t(A\hskip-.015in .) &  HC_*^a(A\hskip-.015in .)\ar[l] &  \underline{HPer}_*^a(A\hskip-.015in .)\ar[l]\ar[u]^{\cong}
}}
\vskip.2in

As before, $B^f$ denotes the topological algebra $B$ equipped with the fine topology (done degreewise if $B$ is simplicial or augmented simplicial). The isomorphism in higher topological $K$-theory follows by Baum's retopologization theorem [Appendix]. The Chern character indicated by $\widetilde{ch}_*^T$ is the lifting of Tillmann's Chern character for fine topological algebras to $\underline{HPer}^t_*(_-)$ (which identifies with $\underline{HPer}^a_*(_-)$ when applied to topological algebras equipped with the fine topology). The projection map $p: A\hskip-.015in .\to \pi_0(A\hskip-.015in .)$ induces an isomorphism $\underline{HPer}_*^a(A\hskip-.015in .)\overset{\cong}{\longrightarrow} \underline{HPer}_*^a(\pi_0(A\hskip-.015in .))$ by the rigidity results of [G]. Finally, the bottom map in the lower left arises from the natural transformation $HC_*^a(_-)\to HC_*^t(_-)$. Each of the homomorphisms in the diagram is functorial in $A\hskip-.015in .$.
\end{proof}

\begin{corollary} For each $x\in H_n(B\pi;\mathbb C)$, with $x, \wt{x}$, and $\Gamma(\wt{x})\hskip-.015in .$ as above, there exist \lq\lq local\rq\rq Connes-Karoubi Chern characters
\begin{gather*}
\wt{ch}^{CK}_*: K_*^t(H^{1,\infty}_L(\pi))\to HC_*^t(H^{1,\infty}_{L\hskip-.015in .}(\Gamma(\wt{x})\hskip-.015in .))\\
\wt{ch}^{CK}_*: K_*^t(\pi_0(H^{CM}_m(\Gamma(\wt{x})\hskip-.015in .)^f))\to HC_*^t(H^{CM}_m(\Gamma(\wt{x})\hskip-.015in .))\\
\wt{ch}^{CK}_*:K_*^t(H^{CM}_L(\pi))\cong  K_*^t(\pi_0(H^{CM}_{L\hskip-.015in .}(\Gamma(\wt{x})\hskip-.015in .)))\to HC_*^t(H^{CM}_m(\Gamma(\wt{x})\hskip-.015in .))
\end{gather*}
\end{corollary}
\vskip.5in


\subsection{Proof of Theorem A}
\vskip.3in
Let $\Gamma\hskip-.03in .$ be a free simplicial resolution of $\pi$. Fix an integral class $0\ne x\in im(H_n(B\pi;\mathbb Z)\to H_n(B\pi;\mathbb C))$, and let $\wt{x}, \Gamma(\wt{x})\hskip-.015in .$ be as defined above. By Proposition \ref{prop:localprops}, there exists an integral fundamental class $\mu_{\tilde{x}}\in im\left(H_n(B\Gamma(\wt{x})\hskip-.015in .;\mathbb Z)\to H_n(B\Gamma(\wt{x})\hskip-.015in .;\mathbb C)\right)$ mapping to $x\in H_n(B\pi;\mathbb C)$ under the map induced by $\Gamma(\wt{x})\hskip-.015in .\hookrightarrow \Gamma\hskip-.03in .\overset{\simeq}{\surj} \pi$, with $\mu_{\tilde{x}}$ canonically represented by an element in $E^1_{1,n-1} = HC_1(\mathbb C[\Gamma(\wt{x})_{n-1}])$ which survives to a non-zero element of $E^{\infty}_{1,n-1}$. Moreover, because $\Gamma(\wt{x})\hskip-.015in .$ is degreewise finitely generated, one may choose a \lq\lq dual\rq\rq integral fundamental class $\mu_{\tilde{x}}^*\in im(H^n(B\Gamma(\wt{x})\hskip-.015in .;\mathbb Z)\to H^n(B\Gamma(\wt{x})\hskip-.015in .;\mathbb C)$ represented by an element\footnote{For $n\ge 3$, the choice of representative $\mu^*_{\tilde{x}}\in HC^1(\mathbb C[\Gamma(\wt{x})_{(n-1)}])$ is essentially unique, while for $n=1,2$ is determined by the canonical basis for $H^1(B\Gamma(\wt{x})_{(n-1)};\mathbb C) \cong HC^1(\mathbb C[\Gamma(\wt{x})_{(n-1)}])$ coming from the generating set for $\Gamma(\wt{x})_{(n-1)}$} in $E_1^{1,n-1} = HC^1(\mathbb C[\Gamma(\wt{x})_{n-1}])\cong H^1(B\Gamma(\wt{x})_{(n-1)};\mathbb C)$ which survives to a non-zero element in $E_{\infty}^{1,n-1}$. Under the pairing $HC^*(\mathbb C[\Gamma(\wt{x})])\otimes HC_*(\mathbb C[\Gamma(\wt{x})])\to\mathbb C$, one has 
\[
<\mu_{\tilde{x}}^*,\mu_{\tilde{x}}> \ne 0
\]

Consider now the commutative diagram
\vskip1in
\hfill(D6)
\vskip-.8in
\centerline{
\xymatrix{
H_{n}(B\Gamma\hskip-.03in .;\mathbb Q)\ar[d]^{\cong}& H_n(B\Gamma(\tilde{x})\hskip-.015in .;\mathbb Q)\ar[l]\ar[d]  \\
 H_{n}(B\pi;\mathbb Q)\ar[d]^{{\cal A}_{\pi}}\ar@{>->}[r]& HC_n^a(\mathbb C[\Gamma(\tilde{x})\hskip-.015in .])\ar[d] \\
K_n^t(H^{1,\infty}_L(\pi))\otimes\mathbb Q\ar[r]^{\widetilde{ch}(\tilde{x})_n^{CK}}
& HC_n^t(H^{1,\infty}_{L\hskip-.015in .}(\Gamma(\tilde{x})\hskip-.015in .))
}}
\vskip.2in
where the arrow in the middle fits into the commuting diagram
\vskip.2in

\centerline{
\xymatrix{
 & K_*^t(\mathbb C[\pi]^f)\otimes\mathbb Q\ar[rd]^{ch_*^{T}} & \\
H_{n}(B\pi;\mathbb Q)\ar[ur]\ar@{>->}[r]\ar[dr] &
\underset{m\in\mathbb Z}{\overset{\phantom{m\in\mathbb Z}}{\oplus}} H_{n-2m}(B\pi;\mathbb Q)\ar@{>->}[r] &\underline{HPer}^a(\mathbb C[\pi])\\ & HC_*^a(\mathbb C[\Gamma(\tilde{x})\hskip-.015in .]) & \underline{HPer}^a_*(\mathbb C[\Gamma(\tilde{x})\hskip-.015in .])\ar[u]^{\cong}\ar[l]
}
}
\vskip.2in

Every rational homology class in $H_n(B\pi;\mathbb Q)$ is a scalar multiple of an integral class, so it suffices to verify injectivity for integral classes. Because the choice of integral class $x\in H_n(B\pi;\mathbb Q)$ is arbitrary, this argument shows that the restricted rationalized assembly map for $K_*^t(\mathbb C[\pi]^f)\otimes\mathbb Q$ is injective; a fact already known [T1]. Our object, then, is to show that the image of $x$ is non-zero in $HC_n^t(H^{1,\infty}_{L\hskip-.015in .}(\Gamma(\tilde{x})\hskip-.015in .))$. As a first step, we have
\vskip.2in

\begin{lemma}\label{lemma:survive} Let $0\ne x\in H_n(B\pi;\mathbb Q)$ be an integral class, with $\Gamma(\wt{x})\hskip-.015in .$ and $0\ne \mu_{\wt{x}}\in E^{1,a}_{1,n-1} = HC_1(\mathbb C[\Gamma(\wt{x})_{n-1}])$ defined as above. Then the image of $\mu_{\wt{x}}$ in $E^{1,t}_{1,n-1} = HC_1(H^{1,\infty}_{L_{st}}\Gamma(\wt{x})_{n-1}))$ 
(under the map induced by the inclusion of simplicial algebras $\mathbb C[\Gamma(\wt{x})\hskip-.015in .]\hookrightarrow H^{1,\infty}_{L\hskip-.015in .}(\Gamma(\wt{x})\hskip-.015in .)$) is non-zero, and survives to a non-zero element in $E^{2,t}_{1,n-1}$ .
\end{lemma}

\begin{proof} As $\Gamma(\tilde{x})_{n-1}$ is free and finitely generated, one has by [JOR1] that
\begin{gather*}
\overline{HC}_*^t(H^{1,\infty}_{L_{st}}(\Gamma(\tilde{x})_{n-1}))_{<1>}  \cong HC_*(\mathbb C[\Gamma(\tilde{x})_{n-1}])_{<1>} = H_*(B\Gamma(\tilde{x})_{n-1};\mathbb C)\\
\overline{HC}_*^t(H^{1,\infty}_{L_{st}}(\Gamma(\tilde{x})_{n-1}))_{<g>} \cong H_*(C_g/(g);\mathbb C) = 0,\quad g\ne 1, n\ge 1
\end{gather*}
For a given group $G$, the map induced by projections onto summands indexed by conjugacy classes 
\[
\overline{HC}_*^t(H^{1,\infty}_{L_{st}}(G))\to \underset{<g>}{\prod} \overline{HC}_*^t(H^{1,\infty}_{L_{st}}(G))_{<g>}
\]
is natural with respect to group homomorphisms. When $G = F$ is a finitely generated free group, the summands indexed by $<g>\ne 1$ vanish in positive dimensions (loc. cit.), resulting in a map
\begin{equation}\label{eqn:functorialsplit}
HC_1^t(H^{1,\infty}_{L_{st}}(F))\to \overline{HC}_1^t(H^{1,\infty}_{L_{st}}(F))\to \overline{HC}_1^t(H^{1,\infty}_{L_{st}}(F))_{<1>} = H_1(BF;\mathbb C)
\end{equation}
which is functorial with respect to homomorphisms of finitely generated free groups. The result is a commuting diagram
\vskip.2in
\centerline{
\xymatrix{
H_1(B\Gamma(\tilde{x})_n)\ar@2{-}[r]\ar[d]^{d_n} &  HC_1(\mathbb C[\Gamma(\tilde{x})_n])\phantom{x}\ar@{>->}[r]\ar[d]^{d^{1,HC}_{1,n}}& 
HC_1^t(H^{1,\infty}_{L_{st}}(\Gamma(\tilde{x})_n))\ar@{->>}[r]\ar[d]^{d^{1,HC}_{1,n}} & H_1(B\Gamma(\tilde{x})_n)\ar[d]^{d_n}\\
H_1(B\Gamma(\tilde{x})_{n-1})\ar@2{-}[r] & HC_1(\mathbb C[\Gamma(\tilde{x})_{n-1}])\phantom{x}\ar@{>->}[r]& 
HC_1^t(H^{1,\infty}_{L_{st}}(\Gamma(\tilde{x})_{n-1}))\ar@{->>}[r] & H_1(B\Gamma(\tilde{x})_{n-1})
}
}
\vskip.2in
by which we conclude that $\mu_{\wt{x}}$ survives to a non-zero class in $E^{2,t}_{1,n-1}$, as claimed.
\end{proof}
To complete the proof, we need to show $\mu_{\wt{x}}$ is not in the image of the differential $d^2_{0,n+1}:E^{2,t}_{0,n+1}\to E^{2,t}_{1,n-1}$. Apriori this is a delicate matter, because, although $E^{1,a}_{0,m}$ is dense in $E^{1,t}_{0,m}$ in the Fr\'echet topology, there is no reason to suppose that this remains true at the $E^2$-level. In fact, the exact relation between $E^{2,a}_{0,*}$ and $E^{2,t}_{0,*}$ for the above complex seems very difficult to determine. This motivates the consideration of an intermediate complex which we now define, first in the algebraic setting, then in the topological one.
\vskip.2in
Let $\Gamma\hskip-.03in .$ be a simplicial group, with $p_n:\Gamma_n\to\pi = \pi_0(\Gamma\hskip-.03in .)$ the canonical projection to $\pi$ (this is realized by any iteration of face maps to $\Gamma_0$, followed by the projection to $\pi$). For $n\ge 0$, we define the homogeneous component of $CC_*(\mathbb C[\Gamma_n])$ to be
\[
CC_*(\mathbb C[\Gamma_n])_h := \underset{p_n(g) = 1}{\underset{<g>\in <\Gamma_n>}{\oplus}} CC_*(\mathbb C[\Gamma_n])_{<g>}
\]
This subcomplex is preserved under face maps as $n$ varies, yielding the simplicial subcomplex
\[
CC_*(\mathbb C[\Gamma\hskip-.03in .])_h :=\{ [n]\mapsto CC_*(\mathbb C[\Gamma_n])_h\}_{n\ge 0}\subset CC_*(\mathbb C[\Gamma\hskip-.03in .])
\]

If $(\Gamma\hskip-.03in .,L\hskip-.015in .)$ is a p-bounded simplicial group with word-length, then  the decomposition of $CC_*^t(H^{1,\infty}_{L_n}(\Gamma_n))$ as a topological direct sum indexed on conjugacy classes (as described in \S 2) allows for the definition of a homogeneous component in an analogous fashion
\[
CC^t_*(H^{1,\infty}_{L_n}(\Gamma_n))_h := \underset{p_n(g) = 1}{\underset{<g>\in <\Gamma_n>}{\widehat{\oplus}}} CC^t_*(H^{1,\infty}_{L_n}(\Gamma_n))_{<g>}
\]
This subcomplex is preserved under face maps as $n$ varies, yielding the simplicial subcomplex
\[
CC^t_*(H^{1,\infty}_{L\hskip-.015in .}(\Gamma\hskip-.03in .))_h :=\{ [n]\mapsto CC^t_*(H^{1,\infty}_{L_n}(\Gamma_n))_h\}_{n\ge 0}\subset CC^t_*(H^{1,\infty}_{L\hskip-.015in .}(\Gamma\hskip-.03in .))
\]
Note that in both the algebraic and topological settings, the collection of projection maps $\{p_n\}_{n\ge 0}$ define (continuous) projections of simplicial (topological) chain complexes
\begin{gather*}
CC_*(\mathbb C[\Gamma\hskip-.03in .])\surj CC_*(\mathbb C[\Gamma\hskip-.03in .])_h\\
CC^t_*(H^{1,\infty}_{L\hskip-.015in .}(\Gamma\hskip-.03in .))\surj CC^t_*(H^{1,\infty}_{L\hskip-.015in .}(\Gamma\hskip-.03in .))_h
\end{gather*}
which are right inverses to the natural inclusions going the other way. 

\begin{lemma}\label{lemma:acyclic} Suppose $\Gamma\hskip-.03in .$ is a free simplicial resolution of $\pi$. Then the associated chain complex of the simplicial vector space $\{[n]\mapsto HC_0(\mathbb C[\Gamma_n])_h\}_{n\ge 0}$ is acyclic above dimension $0$. If $(\Gamma\hskip-.03in .,L\hskip-.015in .)$ is a type $P$ resolution of $(\pi,\ov{L})$, then the same is true for the associated chain complex of the simplicial topological vector space $\{[n]\mapsto HC_0(H^{1,\infty}_{L_n}(\Gamma_n))\}_{n\ge 0}$.
\end{lemma}

\begin{proof} We consider the algebraic case first. For such $\Gamma\hskip-.03in .$, the simplicial chain complex $CC_*(\mathbb C[\Gamma\hskip-.03in .])_h$ is of resolution type, resolving $CC_*(\mathbb C[\pi])_{<1>}$. However, this latter chain complex is also resolved by the simplicial chain complex $\{[n]\mapsto CC_*(\mathbb C[\Gamma_n])_{<1>}\}_{n\ge 0}$, which is a simplicial subcomplex of $CC_*(\mathbb C[\Gamma\hskip-.03in .])_h$. This implies the total complex of the quotient simplicial complex
\[
\ov{CC}_*(\mathbb C[\Gamma\hskip-.03in .])_h := \{[n]\mapsto \ov{CC}_*(\mathbb C[\Gamma_n])_h := {CC}_*(\mathbb C[\Gamma_n])_h/{CC}_*(\mathbb C[\Gamma_n])_{<1>}\}_{n\ge 0}
\]
is acyclic. Because $\Gamma\hskip-.03in .$ is degreewise free, in the simplicial spectral sequence converging to the homology of the total complex of $\ov{CC}_*(\mathbb C[\Gamma\hskip-.03in .])_h$ one has
\[
E^1_{p,q} = 
\begin{cases} 0\quad p > 0\\ \ov{HC}_0(\mathbb C[\Gamma_q])_h\quad p = 0
\end{cases}
\]
where $\ov{HC}_0(\mathbb C[\Gamma_q])_h = HC_0(\mathbb C[\Gamma_q])_h/HC_0(\mathbb C[\Gamma_q])_{<1>}$. But the complex $\{HC_0(\mathbb C[\Gamma_q])_{<1>}\cong \mathbb C\}_{q\ge 0}$ is acyclic above dimension $0$, with $H_0 =\mathbb C$, implying the stated result for the algebraic case.

When $\Gamma\hskip-.03in .$ is a type $P$ resolution of $\pi$, the same argument works for the simplicial $\ell^1$-rapid decay algebra $H^{1,\infty}_{L\hskip-.015in .}(\Gamma\hskip-.03in .)$. Namely, both the simplicial topological complex 
$CC^t_*(H^{1,\infty}_{L\hskip-.015in .}(\Gamma\hskip-.03in .))_h$ as well as its simplicial subcomplex $CC^t_*(H^{1,\infty}_{L\hskip-.015in .}(\Gamma\hskip-.03in .))_{<1>} := \{[n]\mapsto CC^t_*(H^{1,\infty}_{L_n}(\Gamma_n))_{<1>}\}_{n\ge 0}$ are resolutions of $CC_*^t(H^{1,\infty}_{\ov{L}}(\pi))_{<1>}$, implying that the quotient simplicial topological complex
\[
\ov{CC}^t_*(H^{1,\infty}_{L\hskip-.015in .}(\Gamma\hskip-.03in .))_h := \{[n]\mapsto \ov{CC}^t_*(H^{1,\infty}_{L_n}(\Gamma_n))_h := {CC}^t_*(H^{1,\infty}_{L_n}(\Gamma_n))_h/{CC}^t_*(H^{1,\infty}_{L_n}(\Gamma_n))_{<1>}\}_{n\ge 0}
\]
is acyclic. Again, the fact $\Gamma\hskip-.03in .$ is degreewise free implies (by [JOR1], [JOR3]) that in the simplicial spectral sequence converging to the homology of the total complex one has
\[
E^1_{p,q} = 
\begin{cases} 0\quad p > 0\\ \ov{HC}^t_0(H^{1,\infty}_{L_q}(\Gamma_q))_h\quad p = 0
\end{cases}
\]
where $\ov{HC}^t_0(H^{1,\infty}_{L_q}(\Gamma_q))_h = HC^t_0(H^{1,\infty}_{L_q}(\Gamma_q))_h/HC^t_0(H^{1,\infty}_{L_q}(\Gamma_q))_{<1>}$. But the complex $\{HC^t_0(H^{1,\infty}_{L_q}(\Gamma_q))_{<1>}\cong \mathbb C\}_{q\ge 0}$ is canonically isomorphic to its algebraic counterpart, and thus acyclic above dimension $0$, with $H_0 =\mathbb C$. The result follows for the topological case.
\end{proof}

\begin{theorem} For all $n\ge 1$ and integral classes $x\in H_n(B\pi;\mathbb Q)$, the class $\mu_{\tilde{x}}\in E^{2,t}_{1,n-1}$ survives to a non-zero element in $E^{3,t}_{1,n-1}$.
\end{theorem}

\begin{proof} We introduce the method of proof by first verifying the statement in the algebraic case. Of course, for the simplicial group algebra one can use direct knowledge of $HC_*(\mathbb C[\pi])$ to conclude that $\mu_{\tilde{x}}$ must survive, because it hits the image of $x\in H_n(B\pi;\mathbb Q)\to H_n(B\pi;\mathbb C)$, which is non-zero. The object is to find a proof of survivability which does not rely on this last fact. We assume $\Gamma\hskip-.03in .$ is a free simplicial resolution of $\pi$, with $\Gamma(\tilde{x})\hskip-.015in .$ as defined above. As previously noted, there is a projection of simplicial complexes
\[
\{[n]\mapsto CC_*(\mathbb C[\Gamma(\tilde{x})_n])\}_{n\ge 0}\surj 
CC^a_{**} := \{[n]\mapsto CC_*(\mathbb C[\Gamma(\tilde{x})_n])_h\}_{n\ge 0}
\]
The latter embeds in the simplicial complex $CC^{a\prime}_{**} := \{[n]\mapsto CC_*(\mathbb C[\Gamma_n])_h\}_{n\ge 0}$. We define a subcomplex $D^a_{**}\subset C^{a\prime}_{**}$ by
\[
D^a_{pq} = 
\begin{cases}
CC^a_{pq}\quad\text{  if } q\le n-1;\\
CC^{a\prime}_{pq}\text{ if either } q\ge n+2, \text{ or } p\ge 1\text{ and }q\ge n+1;\\
CC^a_{0(n+1)} + d^{0,a}_{0,n+2}(CC^{a\prime}_{0(n+2)}) + b_1(CC^{a\prime}_{1(n+1)})\text{ when } (p,q) = (0,n+1);\\
CC^a_{pn} +  d^0_p(n+1)(D^a_{p(n+1)})\text{ for } q=n, D^a_{p(n+1)}\text{ defined as above}.\\
\end{cases}
\]
It follows by Lemma \ref{lemma:acyclic} that in the spectral sequence converging to $H_*(D^a_{**})$ one has $E^2_{0,n+1} = 0$; in other words we have killed that group. But in the same spectral sequence it is easily seen that the inclusion $C^a_{**}\inj D^a_{**}$ induces an isomorphism on the $E^1_{1,n-1}$-term, as well as on $im(d^1_{1,n}:E^1_{1,n}\to E^1_{1,n-1})$ (note that the actual $E^1_{1,n}$-terms may differ, but the difference comes from the contribution of $d^0_{1(n+1)}(CC^{a\prime}_{1(n+1)})$, which must vanish under $d^1_{1,n}$). Thus the argument that $\mu_{\tilde{x}}$ survives to a non-zero element in $E^2_{1,n-1}$-term of the spectral sequence for $H_*(C^a_{**})$ implies the same for the corresponding spectral sequence for $H_*(D^a_{**})$. However, in this spectral sequence, the vanishing of the $E^2_{0,n+1}$-term implies $E^2_{1,n-1} = E^3_{1,n-1}$. Thus $\mu_{\tilde{x}}$ must survive to a non-zero element in this spectral sequence, in turn implying that it must also for $E^3_{1,n-1}$-term of the spectral sequence for $H_*(C^a_{**})$ from which it came.
\vskip.2in

We now extend this argument to the $\ell^1$-rapid decay case, following essentially the same procedure. For this case we assume $\Gamma\hskip-.03in .$ is a type $P$ resolution of $\pi$. The first step, as before, is to map by the projecton 
\[
\{[n]\mapsto CC^t_*(H^{1,\infty}_{L_n}(\Gamma(\tilde{x})_n))\}_{n\ge 0}\surj 
CC^t_{**} := \{[n]\mapsto CC^t_*(H^{1,\infty}_{L_n}(\Gamma(\tilde{x})_n))_h\}_{n\ge 0}
\]
The latter embeds in the simplicial complex $C^{t\prime}_{**} := \{[n]\mapsto CC^t_*(H^{1,\infty}_{L_n}(\Gamma_n))_h\}_{n\ge 0}$.  As in the algebraic case, we define a subcomplex $D^t_{**}\subset {CC^{a\prime}}_{**}$ by
\[
D^t_{pq} = 
\begin{cases}
CC^t_{pq}\quad\text{  if } q\le n-1;\\
CC^{t\prime}_{pq}\text{ if either } q\ge n+2, \text{ or } p\ge 1\text{ and }q\ge n+1;\\
CC^t_{0(n+1)} + d^{0,a}_{0,n+2}(CC^{t\prime}_{0(n+2)}) + b_1(CC^{t\prime}_{1(n+1)})\text{ when } (p,q) = (0,n+1);\\
CC^t_{pn} +  d^0_p(n+1)(D^t_{p(n+1)})\text{ for } q=n, D^t_{p(n+1)}\text{ defined as above}.\\
\end{cases}
\]
and argue exactly as before. Namely, in the spectral sequence converging to $H_*(D^t_{**})$ one has $E^2_{0,n+1} = 0$. Moreover, the inclusion $C^t_{**}\inj D^t_{**}$ induces an isomorphism on the $E^1_{1,n-1}$-term, as well as $d^1_{1,n}(E^1_{1,n})$. By Lemma \ref{lemma:survive}, $\mu_{\tilde{x}}$ survives to a non-zero element in $E^2_{1,n-1}$-term of the spectral sequence for $H_*(C^t_{**})$, which implies the same for the corresponding spectral sequence converging to $H_*(D^a_{**})$. However, as in the algebraic case, the vanishing of the $E^2_{0,n+1}$-term for this latter spectral sequence implies $E^2_{1,n-1} = E^3_{1,n-1}$. Thus $\mu_{\tilde{x}}$ survives to a non-zero element in this spectral sequence, implying that it must also for $E^3_{1,n-1}$-term of the spectral sequence for $H_*(C^t_{**})$.
\end{proof}
\vskip.2in
With this the proof of Theorem A is complete.
\vskip.5in


\subsection{Related results}
\vskip.3in

Recall that for any free simplicial group $\Gamma\hskip-.03in .'$ we have an equality $H^*(B\Gamma\hskip-.03in .';\mathbb Z) = H^*(D^*(\Gamma\hskip-.03in .',\mathbb Z))$, with the cohomology class $\mu^*_{\tilde{x}}$ represented by an integral cocycle in the (co)complex $D^*(\Gamma\hskip-.03in .',\mathbb Z)$. Fix a minimal simplicial abelian group model\footnote{Here we take $K(\mathbb Z,n-1)\hskip-.01in .$ to be the simplicial free abelian group which is trivial in dimensions $j < (n-1)$, is $\mathbb Z$ in dimension $(n-1)$, and in dimension $m > n$ is the free abelain group on generating set indexed by the iterated degeneracies from dimension $n$ to dimension $m$.} $K(\mathbb Z,n-1)\hskip-.01in .$ for the Eilenberg-MacLane space $K(\mathbb Z,n-1)$. If $\Gamma\hskip-.03in .'$ is not only degreewise free but also finitely generated, then using the cocomplex $D^*(\Gamma\hskip-.03in .',\mathbb Z)$ as in [O1, Appendix] we may realize any integral class $[c]\in H^n(B\Gamma\hskip-.03in .';\mathbb Z)$ via a simplicial group homomorphism
\[
\phi_c: \Gamma\hskip-.03in .'\to K(\mathbb Z,n-1)\hskip-.015in .
\]
which, by finite generation, is a (linearly bounded) homomorphism of simplicial objects in $(f.p.groups)$. The fundamental class $\iota^n\in H^n(BK(\mathbb Z,n-1)\hskip-.015in .;\mathbb C)\cong\mathbb C$ is represented by the generator of $H^1(B\mathbb Z;\mathbb C) = HH^1(\mathbb C[K(\mathbb Z,n-1)_{n-1}])_{<1>}\inj HH^1(\mathbb C[K(\mathbb Z,n-1)_{n-1}]) = E_{1,HH}^{1,n-1}$ which represents a permanent cycle in the Hochschild cohomology spectral sequence converging to $HH^*(\mathbb C[K(\mathbb Z,n-1)\hskip-.015in .])$. That it is non-zero follows from the fact it pairs non-trivially with the canonical dual generator $\iota_n\in H_1(B\mathbb Z;\mathbb C) = HH_1(\mathbb C[K(\mathbb Z,n-1)_{n-1}])_{<1>}\inj HH_1(\mathbb C[K(\mathbb Z,n-1)_{n-1}]) = E^{1,HH}_{1,n-1}$ representing a permanent non-zero cycle in the Hochschild homology spectral sequence converging to $HH_*(\mathbb C[K(\mathbb Z,n-1)\hskip-.015in .])$.
\vskip.2in

In particular, for any integral class $x\in H_n(B\pi;\mathbb Z)$, the complexified cohomology class $\mu^*_{\tilde{x}}\in H^n(B\Gamma(\tilde{x})\hskip-.015in .;\mathbb C)$ admits a representation by a simplicial group homomorphism
\[
\phi_{\tilde{x}}: \Gamma(\tilde{x})\hskip-.015in .\to K(\mathbb Z,n-1)\hskip-.015in .
\]
with $\phi_{\wt{x}}^*(\iota^n) = \mu_{\wt{x}}^*$.
In order to determine the effect of this simplicial homomorphism in algebraic and topological Hochschild cohomology, we will need

\begin{lemma}\label{lemma:equiv} Fix $n\ge 2$. Set $A\hskip-.015in . = \mathbb C[K(\mathbb Z,n-1)\hskip-.015in .], B\hskip-.015in . = H^{1,\infty}_{L\hskip-.015in .}(K(\mathbb Z,n-1)\hskip-.015in .), C\hskip-.015in . = H^{CM}_m(K(\mathbb Z,n-1)\hskip-.015in .)$. Each of these algebras is augmented over $\mathbb C$; denote by $\widehat{A\hskip-.015in .}, \widehat{B\hskip-.015in .},\widehat{C\hskip-.015in .}$ the respective completions with respect to the powers of their augmentation ideal. The passage to completions yields a commuting diagram of  simplicial weak equivalences and isomorphims
\vskip.2in
\centerline{
\xymatrix{
A\hskip-.015in .\ar@{>->}[r]\ar[d]^{\simeq} & B\hskip-.015in . \ar@2{-}[r]\ar[d]^{\simeq} & C\hskip-.015in .\ar[d]^{\simeq}\\
\widehat{A\hskip-.015in .}\ar@2{-}[r] & \widehat{B\hskip-.015in .}\ar@2{-}[r] & \widehat{C\hskip-.015in .}
}
}
\vskip.2in
\end{lemma}

\begin{proof} As $K(\mathbb Z,n-1)\hskip-.015in .$ is a finitely generated free abelian group in each degree, it is degreewise of polynomial growth. For any group with word length $(G,L_{st})$ of polynomial growth, one has an equality $H^{1,\infty}_{L_{st}}(G) = H^{CM}_m(G)$, resulting in the equalities $B\hskip-.015in . = C\hskip-.015in .$, $\widehat{B\hskip-.015in .} = \widehat{C\hskip-.015in .}$ The equivalence $A\hskip-.015in .\to \widehat{A\hskip-.015in .}$ follows by Curtis convergence, as the augmentation ideal is $0$-reduced. If $G$ is a finitely generated free abelian group of rank $N$, $I[G] := ker(\mathbb C[G]\to \mathbb C), I(G) := ker(H^{1,\infty}_{L_{st}}(G)\to\mathbb C)$, then one has short-exact sequences
\begin{gather*}
I[G]^2\inj I[G]\surj \mathbb C^N\\
I(G)^2\inj I(G)\surj \mathbb C^N
\end{gather*}
The first sequence is purely algebraic, and can be viewed in either the discrete or fine topology. The exactness of the second sequence $I(G)^2\inj I(G)\surj \mathbb C^N$ implies $I(G)^2$ is of finite codimension in $I(G)$, hence closed in the Fr\'echet topology. Moreover, the finite dimensionality of $C^N$ implies that the quotient topology on $I(G)/I(G)^2$ induced by the Fr\'echet topology on $I(G)$ via the projection is the same as the fine topology. The result is that not only is the process of competion $B\hskip-.015in .\mapsto \widehat{B\hskip-.015in .}$ well-defined in the Fr\'echet topology (so that Curtis convergence applies here as well), but that one has an equality of simplicial topological algebras $\widehat{A\hskip-.015in .} = \widehat{B\hskip-.015in .}$, where both are topologized degreewise by the fine topology.
\end{proof}

\begin{corollary}\label{cor:isom} The maps $A\hskip-.015in .\inj B\hskip-.015in . = C\hskip-.015in .$ in the previous Lemma induce isomorphisms in both topological Hochschild and topological cyclic (co)homology, where $A\hskip-.015in .$ is topologized degreewise by the fine topology, $B\hskip-.015in .$ and $C\hskip-.015in .$ degreewise by the Fr\'echet topology.
\end{corollary}
\vskip.2in

\begin{proposition}\label{prop:Hochschild} Let $\mu_{\tilde{x}}$ be the integral class defined as above. Then it maps to a permanent cycle in
\[
E^{1,HH,t}_{1,n-1} := HH_1^t(H^{CM}_m(\Gamma(\tilde{x})_{n-1})) 
\]
which survives to a non-zero element $E^{\infty,HH,t}_{1,n-1}$, the $E^{\infty}_{1,n-1}$-term of the Hochschild homology spectral sequence converging to
$HH_*^t(H^{CM}_m(\Gamma(\tilde{x})\hskip-.015in .))$.
\end{proposition}

\begin{proof} The element $\mu_{\tilde{x}}$ comes from an integral permanent cycle in $E^{1,HH,a}_{1,n-1} = HH_1(\mathbb C[\Gamma(\tilde{x})_{n-1}]) = H_1(B\Gamma(\tilde{x})_{n-1};\mathbb C) = HH_1(\mathbb C[\Gamma(\tilde{x})_{n-1}])_{<1>}$, and so is canonically represented by a permanent cycle in the Hochschild homology spectral sequence converging to $HH_*^t(H^{CM}_m(\Gamma(\tilde{x})\hskip-.015in .))$. Now let $\mu^*_{\tilde{x}}$ be an integral dual class for $\mu_{\tilde{x}}$ with representative $\phi_{\tilde{x}}:\Gamma(\tilde{x})\hskip-.015in .\to K(\mathbb Z,n-1)\hskip-.015in .$ This homomorphism induces a continuous homomorphism of simplicial Fr\'echet algebras
\[
H^{CM}_m(\Gamma(\tilde{x})\hskip-.015in .)\to H^{CM}_m(K(\mathbb Z,n-1)\hskip-.015in .)
\]
which, on the level of $E^1$-terms, sends $\mu_{\tilde{x}}'\in HH_1^t(H^{CM}_m(\Gamma(\tilde{x})_{n-1}))$ to a non-zero scalar multiple of the fundamental class in $\iota_n\in HH_1^t(H^{CM}_m(K(\mathbb Z,n-1)_{n-1}))$. By Lemma \ref{lemma:equiv} and Cor.\ \ref{cor:isom}, the inclusion of simplicial (topological) algebras $\mathbb C[K(\mathbb Z,n-1)\hskip-.015in .]\hookrightarrow H^{CM}_m(K(\mathbb Z,n-1)\hskip-.015in .)$ induces an isomorphism
\[
HH_*(\mathbb C[K(\mathbb Z,n-1)\hskip-.015in .])\overset{\cong}{\longrightarrow} HH_*^t(H^{CM}_m(K(\mathbb Z,n-1)\hskip-.015in .))
\]
implying $\iota_n$ survives to a non-zero element in the $E^{\infty}_{1,n-1}$-term of the spectral sequence converging to the group on the right (as it does so on the left). This, in turn, verifies the same property for the cycle $\mu_{\tilde{x}}'$.
\end{proof}
\vskip.2in

The lifting provided by the local Chern character $\wt{ch}(\wt{x})_*^{CK}$ allows for a different formulation of the above results along the lines of [O2].

\begin{theorem} Let $0\ne x\in H_n(B\pi;\mathbb Q), n\ge 1$ be an integral class, with $\tilde{x}$, $\Gamma(\tilde{x})\hskip-.015in .$, and fundamental class $[\mu_{\tilde{x}}]\in H_n(\Gamma(\tilde{x})\hskip-.015in .;\mathbb Q)$ as defined above.Then 
\begin{itemize}
\item For all finitely presented  $\pi$, $0\ne x\in H_n(B\pi;\mathbb Q), n\ge 1$, and choice of integral fundamental class $[\mu_{\tilde{x}}]$, the image of $[\mu_{\tilde{x}}]$ under the composition
\[
\Phi_n(\tilde{x}):H_n(\Gamma(\tilde{x})\hskip-.015in .;\mathbb Q)\inj HH_n(\mathbb C[\Gamma(\tilde{x})\hskip-.015in .])\to HH_n^t( H^{CM}_m(\Gamma(\tilde{x})\hskip-.015in .))
\]
is non-zero.
\item If the image of $[\mu_{\tilde{x}}]$ remains non-zero under the composition 
\[
I_n^t\circ \Phi_n(\tilde{x}): H_n(\Gamma(\tilde{x})\hskip-.015in .;\mathbb Q)\to
HH_n^t(H^{CM}_m(\Gamma(\tilde{x})\hskip-.015in .))\to
 HC_n^t(H^{CM}_m(\Gamma(\tilde{x})\hskip-.015in .))
\]
for all finitely presented $\pi$ and $0\ne x\in H_n(B\pi;\mathbb Q), n\ge 1$, then SNC($\pi$) is true for all finitely presented groups $\pi$.
\end{itemize}
\end{theorem}


\newpage
\section[Appendix]{Appendix 1: $K$-theory of fine topological algebras and the Chern character 
${ch}_*^{CKT}$.}
\vskip.2in

We review the definition of topological $K$-theory for fine topological 
algebras, and prove the retopologization theorem of Baum as it is needed 
in our context. We then show how it applies to yield the lifting of the 
Connes-Karoubi Chern character used in section 4 of this paper.
\vskip.2in

Let $V$ be a vector space over $\Bbb C$. The fine
topology on $V$ is by definition the inductive limit topology on $V$ with 
respect to the family of all one-dimensional (complex) subspaces, where 
each such subspace is taken with the standard topology. In other words, a 
subspace $U\subset V$ is closed iff the intersection $U\cap W$ is closed 
in $W$ for every finite-dimensional subspace $W$ of $V$. This definition 
is due to Grothendieck. If $A$ is an algebra over $\Bbb C$, then $A^f$ will denote 
$A$ equipped with the fine topology.
\vskip.2in

If $A$ is a topological algebra over $\Bbb C$, we let $S_q(A)$ denote the 
algebra of $C^{\infty}$ functions\footnote{the $C^{\infty}$ condition on the maps is needed for the 
construction of the Chern character. The topological $K$-groups remain unchanged if we 
replace  $C^{\infty}$ maps with continuous maps} from the standard $q$-simplex to $A$.
Then $\{[q]\mapsto S_q A\}_{q\ge 0}$ is a simplicial algebra, which we 
equip degreewise with the discrete topology. The
topological $K$-theory space $K(A)$ of $A$ is defined as 
\[
K(A) = |[q]\mapsto BGL(S_q A)|^+
\tag{A.1}
\]
where
\lq\lq$^+$\rq\rq\; denotes Quillen's plus construction (compare [T, Def. 
1.1]). It is often the case that the topological monoid 
$GL(A)$ is open in $M(A)$ (for example, when $A$ is a Banach algebra). In 
this case, the space in (A.1) is homotopy equivalent to 
$|[q]\mapsto BGL(S_q A)|$. If $GL(A)$ is a topological group, then 
$|[q]\mapsto BGL(S_q A)|$ is homotopy 
equivalent to $BGL(A)$, in which case the space in (A.1) is homotopy 
equivalent to $BGL(A)$.
Again, Banach algebras are a good example of when these equivalences hold.
The higher topological $K$-groups of $A$ are defined  
by $K_*^t(A) = \pi_*(K(A))$ for $* > 0$. We write $GL.(A)$ for the 
simplicial group 
$\{[q]\mapsto GL(S_q A)\}_{q\ge 0}$.
\vskip.2in

We will want to know the effect of 
retopologization on topological $K$-theory. In other words, if $A$ is an 
algebra equipped with two (possibly distinct) topologies, denoted by 
$A^{T^1}$ and $A^{T^2}$ with $T^1$ finer than $T^2$, then the identity map on $A$ induces a continuous map
\[
A^{T^1}\to A^{T^2}
\]
and so an induced map on topological $K$-groups
\[
K^t_*(A^{T^1})\to K^t_*(A^{T^2})
\]
\begin{theorem}(Baum) If $A$ is a (not necessarily complete) normed algebra, 
$T_1$, $T_2$ two topologies on $A$ for which the identity map on objects 
induces continuous maps
\[
A^f\to A^{T_1}\to A^{T_2}\to A
\]
(where $A$ by itself means that $A$ is taken with the norm topology), and 
$GL(A)$ is open in $M(A)$, then
the map $A^{T_1}\to A^{T_2}$ induces a weak equivalence
\[
GL(A^{T_1})\overset\simeq\to\to GL(A^{T_2})
\]
\end{theorem}

\begin{proof} Let $T$ denote a topology finer than the norm topology. Set $S(A) = 
M(A) - GL(A)$. Then $S(A)$ is 
closed in $M(A)$ in the norm topology. Let $K$ be a finite simplicial 
complex, and $f : K\to GL(A)$ a continuous
map. As $S(A)$ is closed, $d = d(f(K),S(A)) >0$. Fix an $\epsilon$ with $0 
< \epsilon < \frac{d}{4}$. 
Via barycentric subdivision, we can assume the triangulation satisfies the 
property that the image of any simplex in  $K$ under $f$ lies within a ball of radius $\epsilon$. Suppose 
given an $m$-simplex $\Delta^m_i$ of $K$, with vertices $x_{0,i},\dots,x_{m,i}$. Set $g_{j,i} = f(x_{j,i})$. 
Define $P(f) : \Delta^m_i\to GL(A)$ in terms of barycentric coordinates by $P(f)(\sum_{j=0}^m\alpha_j x_{j,i}) = 
\sum_{j=0}^m\alpha_j g_{j,i}$. This map is continuous in the fine topology, and the construction is compatible on 
intersections of simplices. For $x\in\Delta^m_i$ with barycentric coordinates as above, and $x_{k,i}$ 
a vertex of $\delta^m_i$ we have 
\[
d(P(f)(x),S(A)) > d(P(f)(x_{k,i}),S(A)) - d(P(f)(x),P(f)(x_{k,i})) > 
d-\epsilon > 0
\]
Also

\begin{gather*}
d(P(f)(x),f(x)) \le d(P(f)(\sum_{j=0}^m\alpha_j x_{j,i}), f(x_{k,i})) + 
d(f(x_{k,i}),f(x))\\
< \sum_{j=0}^m\alpha_j\ d(P(f)(x_{j,i}), f(x_{k,i})) + \epsilon < 2\epsilon
\end{gather*}

So this piece-wise linear extension
defines a continuous map $P(f) : K\to GL(A^T)$, and the same distance 
argument implies that
the linear homotopy on $K\times [0,1]$ given by
$h(x,t) = tP(f)(x) + (1-t)f(x)$ yields a continuous homotopy $h : K\times 
[0,1]\to GL(A)$.
Thus the induced map on homotopy groups $\pi_*(GL(A^T))\to \pi_*(GL(A))$ 
is surjective.
In addition, we note that $P(P(f)) = P(f)$. Therefore, if $K' \subset
K$ and $g = f|_{K'}$, then $P(g) = P(f)|_{K'}$. Thus homotopies may be 
lifted in a way compatible with
the lifting of their restriction at the two ends, implying the 
injectivity of $\pi_*(GL(A^T))\to \pi_*(GL(A))$.
\end{proof}
\vskip.02in

\begin{remark} This result applies, for example, when the algebra $A$ is 
holomorphically closed in a Banach algebra, 
for in this case $GL(A)$ is open in $M(A)$ in the (induced) norm topology. 
In particular, the result applies to the rapid decay subalgebra 
$H^{1,\infty}_L(\pi)$ of $\ell^1(\pi)$. 
\end{remark}
\vskip.1in

A technical point arises here, for in order to apply Baum's theorem as 
it stands to topological $K$-theory, we would need the natural inclusion 
\[
GL.(A) = \{GL(S_qA)\}_{q\ge 0}\inj S.GL(A) = \{S_qGL(A)\}{q\ge 0}
\tag{A.3}
\]
to be an equivalence. When the map 
$g\mapsto g^{-1}$ is continuous in the topology on $GL(A)$, the map in 
(A.3) is an isomorphism of simplicial sets. The following 
variation provides the version of Baum's theorem that  
applies to our situation; the proof is essentially the same.

\begin{theorem} (Baum variation) If $A$ is a normed 
algebra, 
$T_1$, $T_2$ two topologies on $A$ for which the identity map on objects 
induces continuous maps
\[
A^f\to A^{T_1}\to A^{T_2}\to A
\]
(where $A$ by itself means that $A$ is taken with the norm topology), and 
$GL(A)$ is open in $M(A)$, then
the map $A^{T_1}\to A^{T_2}$ induces a weak equivalence of simplicial 
groups
\[
GL.(A^{T_1}) = \{GL(S_q A^{T_1})\}_{q\ge 0} \overset\simeq{\longrightarrow} 
GL.(A^{T_2}) = \{GL(S_q A^{T_2})\}_{q\ge 0}
\]
\end{theorem}

\begin{proof} As before, it suffices to prove the theorem when $T_2$ is the norm 
topology and $T = T_1$ any finer topology. Since $GL.(A)$ is a Kan 
complex, an element of  $\pi_n(GL.(A))$ is 
represented by an element $f= \{f_{jk}\}\in GL(S_n A)$ with 
$\partial_i(f_{jk}) = *$ for all $j,k$ and
$0\le i\le n$. As above, let $d = d(im(f),S(A))$ (measured via the 
embedding  $GL(S_n A)\inj S_nGL(A)$) and $\varepsilon = \frac{d}{4}$. 
The domain of $f_{jk}$ is $S^n = \Delta_n/\partial(\Delta_n)$, and only 
finitely many of the $f_{jk}$ are non-constant. By barycentric subdivision 
we may 
replace $f_{jk} : S^n\to A$ by $\wt f_{jk} : K\to A$ where $|K|\cong 
S^n$, $\wt f_{jk}$ is the corresponding refinement of $f_{jk}$ and for 
each simplex $\Delta_i\in K$, the image $\wt f(\Delta_i) = \{\wt 
f_{jk}(\Delta_i)\}_{j,k}$\; 
lies within a ball of radius $\varepsilon$. The rest follows as before 
in the proof of Baum's theorem. Namely, define $P(\wt f)$ as the matrix 
$\{P(\wt f_{jk})\}$. This $P$ is again compatible on intersections of 
simplices, is an idempotent operation, and is continuous in the fine 
topology. The distance condition guarantees that the image of $P(\wt f)$ 
lies in the $n$-skeleton of $GL.(A)$, and is homotopic to $\wt 
f$ via a linear homotopy which also lifts by $P$. This implies the 
surjectivity and injectivity of the induced map on homotopy groups.
\end{proof}
\vskip.1in

If $GL(A)$ is open in $M(A)$, then $Id:A^{T_1}\to A^{T_2}$ induces an isomorphism on 
topological $K_0$-groups, because 
in this case $K_0$ is insensitive to the topology on $A$. Baum's theorem 
may be viewed as the analogue of this isomorphism for higher topological 
$K$-theory.
\vskip.2in

If $A$ is a locally convex topological algebra, and $CC_*^t(A)$ the 
completion of the cyclic complex with respect to the topology on $A$, then 
the homology of this complex yields the \underbar{unreduced} topological 
cyclic homology groups $HC^t_*(A)$, while forming the quotients 
$ker(d^t_*)/\ov{im(d^t_{*+1})}$ yield the \underbar{reduced} topological 
cyclic homology groups $\ov{HC}^t_*(A)$ of $A$. The Connes-Karoubi Chern 
character, constructed by Karoubi in [K], is a graded homomorphism of 
abelian groups
\[
K_*^t(A)\overset{ch_*^{CK}}\longrightarrow \ov{HC}^t_*(A)
\]

Tillmann has shown in [T] that Karoubi's framework for constructing a Chern character on the topological $K$-groups of Fr\'echet algebras carries over to the fine 
topology. The advantage to using this topology in a homological setting is 
clear, for in the fine topology the algebraic tensor product is complete, 
and $im(d_{*+1})$ is closed in $ker(d_*)$. The result of [T] then is a 
Chern character
\[
K_*^t(A^f)\overset{ch^{T}_*}\longrightarrow HC_*(A^f)
\]
where $A^f$ denotes $A$ with the fine topology and $HC_*(A^f)$ its 
\underbar{algebraic} cyclic homology groups. Tillmann also showed that the 
Chern character $ch^f_*$ was compatible with the Connes-Karoubi Chern 
character in the sense that the diagram
\vskip.2in
\centerline{
\diagram
K_*^t(A^f)\rto \dto^{ch^f} & K_*^t(A)\dto^{ch_*^{CK}} \\
HC_*(A^f)\rto & \ov{HC}_*^t(A)
\enddiagram
}
\vskip.2in
commutes. Combining this result with the previous theorem yields

\begin{theorem} If $A$ is a locally convex (not necessarily complete) normed
topological algebra, then the Connes-Karoubi Chern character for $A$ 
factors as
\[
K^t_*(A)\cong K_*^t(A^f)\overset{ch^{T}_*}\longrightarrow HC_*^t(A^f)\to {HC}_*^t(A)\to \ov{HC}_*^t(A)
\]
\end{theorem}

The composition of the first two maps $K_*^t(A)\to HC_*^t(A^f)$ is denoted $ch_*^{CKT}$.
\vskip.2in

The definition of the topological $K$-theory space given here is appropriate for higher $K$-groups starting in dimension 1, but not so for non-positively graded $K$-groups. This is not a concern for the applications to this paper, because there is no place where negative groups are used, and only one argument for which $K_0(_-)$ is needed: this is in the verification of the isomorphism
\begin{equation}\label{eqn:isoCM}
K_*^t(\pi_0(H^{CM}_{L\hskip-.015in .}(\Gamma\hskip-.03in .)^f))\cong K_*^t(C^*_m(\pi)),\quad *\ge 1
\end{equation}
where $\Gamma\hskip-.03in .$ is a degreewise finitely generated free simplicial group with $\pi_0(\Gamma\hskip-.03in .) = \pi$. We observe that 
\begin{itemize}
\item $\pi_0(H^{CM}_{L\hskip-.015in .}(\Gamma\hskip-.03in .)^f) = H^{CM}_{L\hskip-.015in .}(\Gamma\hskip-.03in .)_0^f/(H^{CM}_{L\hskip-.015in .}(\Gamma\hskip-.03in .)_0^0)^f$,
\item $H^{CM}_{L\hskip-.015in .}(\Gamma\hskip-.03in .)_0 = H^{CM}_{L_0}(\Gamma_0)$ is dense and holomorphically closed in $C^*_m(\Gamma_0)$
\item $H^{CM}_{L\hskip-.015in .}(\Gamma\hskip-.03in .)_0^0$ is dense and holomorphically closed in $C^*_m(\Gamma\hskip-.03in .)^0_0$
\end{itemize}

\begin{proposition} Let $I\inj A\surj \ov{A}$ be a short-exact sequence of Banach algebras, $B\subset A$ a dense subalgebra of $A$ closed under holomorphic functional calculus, and $J\subset I$ an ideal in $B$ which is dense and holomorphically closed in $I$\footnote{We do not assume here that $J$ is closed in $B$ in any topology coarser than the fine topology.}. Let $\ov{B} = B/J$, topologized with the fine topology. Then there is an exact sequence
\[
K_1^t(B^f)\to K_1(\ov{B})\to K_0^t(J^f)\to K_0^t(B^f)\to K_0^t(\ov{B})
\]
\end{proposition}

\begin{proof} For a topological algebra $D$ (possibly without unit), $K_1^t(D)$ as defined above may be realized as the cokernel of the map
\[
K_1^a(S_1(D))\overset{d_1}{\to} K_1^a(S_0(D))
\]
where $d_1 := (\partial_1)_* - (\partial_0)_*$. We also observe that if $D$ is dense and holomorphically closed in a Banach algebra $D'$, then $K_0^a(D) = K_0^t(D)$ for any topology on $D$ between the fine topology and the norm topology induced by the inclusion in $D'$. Consider the commuting diagram
\vskip.2in
\centerline{
\xymatrix{
K_1^a(S_1(B^f))\ar[d]^{d_1}\ar[r] & K_1^a(S_1(\ov{B}))\ar[d]^{d_1}\ar[r]^{\partial} &
K_0^a(S_1(J^f))\ar[d]^{d_1 = 0}\ar[r] & K_0^a(S_1(B^f))\ar[d]^{d_1 = 0}\ar@{->>}[r] & \ov{K}_0^a(S_1(\ov{B}))\ar[d]^{d_1=0}\\
K_1^a(S_0(B^f))\ar@{->>}[d]\ar[r] & K_1^a(S_0(\ov{B}))\ar@{->>}[d]\ar[r]^{\partial} &
K_0^a(S_0(J^f))\ar[d]^{\cong}\ar[r] & K_0^a(S_0(B^f))\ar[d]^{\cong}\ar@{->>}[r] & \ov{K}_0^a(S_0(\ov{B}))\ar[d]^{\cong}\\
K_1^t(B^f)\ar[r] & K_1^t(\ov{B})\ar[r]^{\partial} &
K_0^t(J^f)\ar[r] & K_0^t(B^f)\ar[r] & \ov{K}_0^t(\ov{B})
}
}
\vskip.2in
Here $\ov{K}^a_0(S_i(\ov{B}))$ denotes the image of $K_0^a(S_i(B^f))$ in $K_0^a(S_i(\ov{B}))$, and similarly with \lq\lq a\rq\rq replaced by \lq\lq t\rq\rq. The first two rows arise from the long-exact sequence in algebraic $K$-theory. A diagram chase then shows that the bottom  sequence is exact.
\end{proof}

\begin{corollary} Under the same conditions as the previous proposition, there is a long-exact sequence
\[
\dots 
\to K_i(J^f)\to K_i^t(B^f)\to K_i(\ov{B})\to K_{i-1}^t(J^f)\to\dots \to K_0^t(B^f)\to K_0^t(\ov{B})
\]
this long-exact sequence induces an isomorphism
\[
K_i^t(\ov{B})\cong K_i^t(\ov{A}),\quad i\ge 1
\]
\end{corollary}
\begin{proof} The conditions on $J$ and $B$ imply that $K(J^f)\simeq hofib(K(B^f)\to K(\ov{B}))$, yielding a long-exact sequence of higher topological $K$-groups, which spliced together with the sequence from the preceding proposition yields the indicated long-exact sequence terminating in $K_0^t(\ov{B})$. The compatible inclusions $J\hookrightarrow I$, $B\hookrightarrow A$ then induce isomorphisms in topological $K$-theory, as well as a map of long-exact sequences of topological $K$-groups. The result follows from the five lemma.
\end{proof}
\newpage



\begin{thebibliography}{XXXX}


\bibitem[B]{b} D. Burghelea, {\it The cyclic homology of the group rings}, Comm. Math {\bf 60} (1985), 354 -- 365.

\bibitem[CM]{cm} A. Connes and H. Moscovici, {\it Hyperbolic Groups and the Novikov Conjecture}, Topology {\bf 29} (1990), 345 -- 388.

\bibitem[dH]{dH}P. de la Harpe, {\it Groupes hyperboliques, alg\`ebres d'op\'erateurs et un th\'eor\`eme de Jolissaint}, C. R. Acad. Sci. Paris, S\'erie 1 {\bf 307} (1988), 771 -- 774.

\bibitem[G]{g} T. Goodwillie, {\it Cyclic homology, derivations, and the free loop space}, Topology {\bf 24} No.2 (1985), 187 -- 215.

\bibitem[Gr]{gr} M. Gromov, {\it Hyperbolic groups}, {\it Essays in group theory}, Math. Sci. Res. Inst. Publ. 8, (Springer-Verlag, New York) (1987), 75 -- 263.

\bibitem[H]{H} U. Haagarup, {\it An example of a non-nuclear $C^*$-algebra which has the approximation property}, Inv. Math. {\bf 50} (1979), 279 -- 293.

\bibitem[HR]{hr} N. Higson and J. Roe, {\it Amenable actions and the Novikov Conjecture}, J. Reine Angew. Math. {\bf 519} (2000), 143 -- 153.

\bibitem[JOR1]{jor1} R. Ji, C. Ogle, and B. Ramsey, {\it Relatively hyperbolic groups, rapid decay algebras and a generalization of the Bass conjecture}, J. Noncommut. Geom. {\bf 4} (2010), 83 -- 124.

\bibitem[JOR2]{jor2} R. Ji, C. Ogle, and B. Ramsey, {\it $\cal B$-Bounded Cohomology and Applications}, preprint (2010).

\bibitem[PJ1]{pj1} P. Jolissaint, {\it K-theory of Reduced $C^*$-algebras and Rapidly Decreasing Functions on Groups}, K-theory {\bf 2} (1989), 723 -- 735.

\bibitem[PJ2]{pj2} P. Jolissaint, {\it Rapidly Decreasing Functions in Reduced $C^*$-algebras of Groups}, Trans. Amer. Math. Soc. {\bf 317} (1990), 167 -- 196.

\bibitem[MK]{mk} M. Karoubi, {\it Homologie Cyclique et K-th\'eorie}, Ast\'erisque {\bf 149} (1987).

\bibitem[GK]{gk} G. Kasparov, {\it Equivariant KK-theory and the Novikov Conjecture}, Inv. Math {\bf 91} (1988), 147 -- 201.

\bibitem[L]{l} J. Loday, {\it Cyclic Homology}, A Series of Comprehensive Studies in Math. {\bf 301} (1997)

\bibitem[M1]{M1} R. Meyer, {\it Combable groups have group cohomology of polynomial growth}, Quart. J. Math. {\bf 57} (2006), pp. 241 -- 261.

\bibitem[O0]{O0} C. Ogle, {\it Bounded Homotopy Theory and the Novikov Conjecture}, OSU preprint (1993 -- 2003).

\bibitem[O1]{O1} C. Ogle {\it Polynomially Bounded Cohomology and Discrete Groups}, Jour. Pure and App. Alg. {\bf 195} (2005), 173 -- 209.

\bibitem[O2]{O2} C. Ogle {\it Filtrations of Simplicial Functors and the Novikov Conjecture}, K-theory {\bf 36} (2005), 345 -- 369.

\bibitem[MP]{MP} M. Puschnigg, {\it The Kadison-Kaplansky conjecture for word-hyperbolic groups}, Inv. Math. {\bf 149} (2002), 153 -- 194.
\bibitem[T]{t} U. Tillmann, {\it K-theory of Fine Topological Algebras, Chern Character, and Assembly}, K-theory {\bf 6}, 57 -- 86.


\end{thebibliography}
\end{document}